\documentclass[11pt]{article} 
\usepackage[letterpaper, margin=1in]{geometry} 
\usepackage{setspace}
\usepackage{hyperref}
\usepackage[ruled]{algorithm2e} 
\usepackage{algorithmic}
\usepackage{amsmath}
\usepackage{amsfonts}
\usepackage{amsthm}
\usepackage{dsfont}
\usepackage{tikz}
\usepackage{caption}
\usepackage{subcaption}
\usepackage{xcolor}
\usepackage{changepage}

\usetikzlibrary{quotes,angles}
\usetikzlibrary{calc}
\usetikzlibrary{decorations.pathreplacing}

\usepackage{notations}

\newcommand{\modelname}{\instancename}
\newcommand{\instancename}{\renewcommand{\algorithmcfname}{MODEL}}

\definecolor{WildStrawberry}{RGB}{255,67,164}

\newtheorem{theorem}{Theorem}[section]
\newtheorem{lemma}[theorem]{Lemma}
\newtheorem{claim}[theorem]{Claim}
\newtheorem{definition}[theorem]{Definition}

\newcommand{\set}[1]{\{#1\}}

\newcommand\Plus{\texttt{+}}
\newcommand\Minus{\texttt{-}}

\allowdisplaybreaks

\title{On the satisfiability of random $3$-SAT formulas\\with $k$-wise independent clauses}
\author{Ioannis Caragiannis \and Nick Gravin \and Zhile Jiang}
\date{}

\begin{document}

\maketitle

\begin{abstract}
    
     The problem of identifying the satisfiability threshold of random $3$-SAT formulas has received a lot of attention during the last decades and has inspired the study of other threshold phenomena in random combinatorial structures. The classical assumption in this line of research is that, for a given set of $n$ Boolean variables, each clause is drawn uniformly at random among all sets of three literals from these variables, independently from other clauses. Here, we keep the uniform distribution of each clause, but deviate significantly from the independence assumption and consider richer families of probability distributions. For integer parameters $n$, $m$, and $k$, we denote by $\DistFamily_k(n,m)$ the family of probability distributions that produce formulas with $m$ clauses, each selected uniformly at random from all sets of three literals from the $n$ variables, so that the clauses are $k$-wise independent. Our aim is to make general statements about the satisfiability or unsatisfiability of formulas produced by distributions in $\DistFamily_k(n,m)$ for different values of the parameters $n$, $m$, and $k$. 
     
     Our technical results are as follows: First, all probability distributions in $\DistFamily_2(n,m)$ with $m\in \Omega(n^3)$ return unsatisfiable formulas with high probability. This result is tight. We show that there exists a probability distribution~$\Distribution \in \DistFamily_3(n,m)$ with $m\in O(n^3)$ so that a random formula drawn from $\Distribution$ is almost always satisfiable. In contrast, for $m=\Omega(n^2)$, any probability distribution~$\Distribution \in \DistFamily_4(n,m)$ returns an unsatisfiable formula with high probability. This is our most surprising and technically involved result. Finally, for any integer $k\geq 2$, any probability distribution~$\Distribution \in \DistFamily_k(n,m)$ with $m=O(n^{1-1/k})$ returns a satisfiable formula with high probability.

\end{abstract}

\section{Introduction}
\label{sec:intro}
Satisfiability of propositional formulas (SAT) is one of the most renowned problems in theoretical computer science.
It appeared in the first lists of NP-complete problems independently proposed by Cook and Levin, and is pivotal for many developments
in modern complexity theory. Today, many lower bounds on the running time of algorithms rely on the Exponential Time Hypothesis for 
solving SAT \cite{Bringmann14, CNPPRW22, IP01, IPZ01}.  On the practical side, SAT solvers are frequently deployed in hardware circuit design, model checking, program verification, automated planning and scheduling, as well as in solving real-life instantiations of combinatorial optimization problems such as FCC spectrum auctions. Modern SAT solvers often find solutions to large industrial instances with thousands or even millions of variables despite the NP-hardness of the problem. However, there is still a large discrepancy between the performance of SAT solvers on those instances and theoretical average-case predictions, which have been studied in great depth under the line of research on \emph{random SAT}. 

\paragraph{Random SAT.} A $j$-CNF formula $\phi$ over $n$ variables is composed of $m$ OR-clauses, 
each containing exactly $j$ literals of $j$ different variables. In the most commonly studied random SAT model, 
a formula $\phi$ is generated uniformly at random from all possible $j$-CNF formulas over $n$ variables and 
$m$ clauses. The most prominent theoretical question related to random SAT is to identify the satisfiability threshold $r_j$
such that $\lim_{n\to \infty}\Prx{\phi \text{ is satisfiable}}$ is equal to $0$ when $m/n > r_j$, and equal to $1$ when $m/n<r_j$.
It has been established~\cite{CR1992} that $2$-SAT has $r_2=1$,  and its phase transition window~\cite{BBCKW2001} is $m\in[n-\Theta(n^{1/3}), n+\Theta(n^{1/3})]$. For $j\ge 3$, the asymptotic $j$-SAT threshold was shown to be $2^j\log 2 -\frac{1}{2}(1+\log 2)\pm o_j(1)$ as $j\to\infty$~\cite{Coja14} (improving previous results from~\cite{AP04}), while for large enough $j$ the exact value of $r_j$ was determined in~\cite{DSS2022}. However, the question of
identifying $r_j$ for small values of $j$ remains open. In particular, random $3$-SAT has attracted a lot of attention. 
For the lower bound part, it has been shown in a series of papers~\cite{CR1992,BFU1993,FS1996,Achlioptas2000,HS03,KKL06} that $r_3\ge 3.52$ (the currently best known bound is due to~\cite{HS03,KKL06}). The upper bound part is studied by \cite{FP1983,KMPS95,DBM2000,DKMP09}; the currently best known bound is $r_3< 4.49$ due to~\cite{DKMP09}. The estimate $r_3\approx 4.26$ was derived from numerical experiments~\cite{MPZ02} (see also~\cite{CKT1991,MSL1992}).

A more recent line of work~\cite{FKRRS2017, FR2018, FR2019, OB2019} extends the standard model of random $j$-SAT to non-uniform distributions.
Their motivation comes from the empirical observation that, in practice, CNF formulas often have rather different frequencies/probabilities for the $n$ variables to appear in each clause (following a power-law distribution instead of a uniform one). Namely, Friedrich and  Rothenberger~\cite{FR2019} proposed a non-uniform random model, where the literals $\{x_i,\overline{x_i}\}_{i\in [n]}$ are  selected independently at random in each clause $c$ of the random $j$-CNF with $\Prx{x_i\in c}=\Prx{\overline{x_i}\in c}=p_i$ and where probabilities $\vect{p}=(p_i)_{i\in [n]}$ may vary across different variables. They find satisfiability threshold $r_2(\vect{p})$  of non-uniform random $2$-SAT for certain regimes depending on $\vect{p}$. However, the non-uniform model of~\cite{FR2019} does not capture the community biases/correlations (i.e., the fact that certain variables are more likely to appear together in a clause), which are often observed in practice \cite{AGL12}. This leads us to the question of whether it is possible to relax the strong independence assumption in the existing random SAT literature.  

\paragraph{Relaxation of independence.}  We first observe that it does not make much sense to study distributions of SAT formulas with arbitrary correlations over the clauses. Indeed, by allowing correlation between several clauses, one may enforce that the random formula $\phi$ contains large fixed sub-formulas corresponding to NP-hard SAT variants. This would be at odds with our goal 
of studying average-case complexity. Therefore, we must keep a certain degree of independence in the distribution of instances. We propose to consider the relaxation of mutual independence over $m$ clauses in a random formula $\phi$ \emph{to $k$-wise independence} for a small constant $k$. To keep the new model tractable, we focus on $3$-SAT and uniform distribution of literals within each clause. I.e., we assume that (i) every $3$-OR-clause $c$ of a random $3$-CNF formula $\phi$ has three literals of three distinct variables drawn uniformly at random among all such triplets of literals and that (ii) given this marginal distribution of each 
clause $c\sim\Dmarginuni$, the distribution $\Distribution$ over the clause set $\ClauseSet$ in $\phi$ is only $k$-wise independent instead of the mutually independent distribution $\Distind=\left(\Dmarginuni\right)^{\otimes m}$ in the standard model. This is a natural generalization that has been considered in a number of different settings but, to the best of our knowledge, not in the context of random SAT. Note that the smaller $k$ is, the bigger the set of possible distributions $\Distribution$. 
Furthermore, for small values of $k$, a $k$-wise independent distribution $\Distribution$ can still capture a large class of dependencies among clauses but at the same time does not allow correlation between any $k$-tuples of clauses. 
In mathematical terms, the family of discrete $k$-wise independent distributions naturally appears when we map the set of distributions to the set of their low-degree moments. Specifically, if a distribution $\Distribution$ is supported on the $n$-dimensional binary cube\footnote{Similar moment functions can be defined for any distribution with discrete marginals.} $\texttt{supp}(D)=\{-1,1\}^{n}$, then all its moments of degree up to $k$ can be described as $\mu(D)=(\Ex{\prod_{i\in S}x_i})_{|S|\le k}$. 
As low-degree moments (basically, the image of $\mu$) are extremely important in statistical analysis, it is equally important to study the kernel of the aforementioned mapping, which exactly corresponds to the family of $k$-wise independent distributions. 
Let us provide additional justifications of our framework by discussing some of the theoretical work on random $3$SAT and on other settings with similar $k$-wise independence relaxation.

\begin{description}
\item[Pseudo-randomness.] Historically, the $k$-wise relaxation of independence has been actively used in the literature on 
		derandomization and pseudo-randomness, as it allows to significantly reduce the amount of random bits needed to generate random objects. 
		For example, Alon and Nussboim~\cite{AlonN08} consider random Erd\H{o}s-R\'{e}nyi graphs and examine the minimal degree $k$ of independence 
		needed to achieve a variety of graph properties and statistics (such as connectivity, existence of perfect matchings, 
		existence of Hamiltonian cycles, clique and chromatic numbers, etc.) that match those in the mutually independent case. Benjamini et al.~\cite{BenjaminiGP12} consider similar questions for monotone boolean functions. The motivation in~\cite{AlonN08} comes from the fact that there are efficient constructions of $k$-wise independent distributions with ``low degree of independence'' (say $k=O(\log n)$) that utilize only $\polylog(n)$ random bits, i.e., much fewer than the polynomial number of random bits required to generate mutually independent distributions. While some of this motivation can be applied to our setting of random $3$-SAT, it is a conceptually different story. Indeed, the perspective of pseudo-random generation is through the lenses of ``probability theory'', where one controls the distributions and can simply choose one that satisfies necessary conditions such as, e.g., $(\log n)$-wise independence. On the other hand, our motivation stems from ``statistics'', as our ideal model should have a reasonable fit to empirical observations. So, we would like to use as minimal assumptions as possible and study small (constant) degrees of independence.
\item[Refutability of $3$-SAT.] While the research on lower bounds for random $3$-SAT often comes up with certain simple heuristics that efficiently find a satisfying assignment (see, e.g., the surveys by Achlioptas~\cite{A21} and Flaxman~\cite{Flaxman16}), it is extremely hard to find an efficient refutation of a unsatisfiable $3$-SAT formula.
		Indeed, the common approach to refute a given SAT formula is proof in resolution. Chvatal and Szemeredi~\cite{ChvatalS88} first showed that a random $3$-CNF formula with $m=\Theta(n)$ clauses (which is almost surely unsatisfiable) almost surely admits only exponential size proof in resolution. Later, Ben-Sasson and Wigderson~\cite{BenSassonW01} derived similar result for much larger $m=O(n^{3/2-\eps})$. On the positive side, \cite{FriedmanGK05} gave the first polynomial time algorithm via spectral techniques that almost surely\footnote{Refutation in this case is an algorithm with one-sided error: it always 
		refutes the formula correctly by producing certain certificates, or says that the formula might be correct.} refutes a random $3$-SAT formula with $m=n^{3/2+\eps}$ clauses. The best known bound on $m$ is due to Feige and Ofek~\cite{FeigeO07} who proved that, for a sufficiently large constant $c$, random $3$-SAT formulas 
		with $m=c\cdot n^{3/2}$ clauses can be almost surely refuted in polynomial time using another spectral graph algorithm. We note that a similar situation (extremely high probability of unsatisfiability for a random formula and inability to efficiently confirm it) is unlikely to happen in our $k$-clause independent model for constant $k$. Indeed, the main proof approach for dealing with arbitrary $k$-wise independent distribution is to define a $k$-wise statistic, which differentiates any satisfiable formula from a typical unsatisfiable one.
\item[Testing $k$-wise independence.] The property of $k$-wise independence of a distribution with $n$ components can be tested using $n^{O(k)}=poly(n)$ many samples in polynomial time, when $k$ is a constant~\cite{AAKMRX07, RX10}. This is a useful property to have, as it allows one to verify with only polynomially many instances of random $3$-SAT, whether these instances conform to $k$-wise independence or not.
\item[Robust mechanism design.] A recent line of work in robust mechanism design also considers families of $k$-wise independent Bayesian priors in single and multi-unit auctions \cite{CGLW21, DKP24, GW24, GHKL24}. 
Their motivation is similar to ours, as they also rely on the statistical point of view to justify the extension of the results for mutual independent priors typically assumed in Bayesian mechanism design to $k$-wise independent ones. 		 		
\end{description}

\subsection{Problem formulation}
\label{sec:formulation}
We consider random $3$-CNF formulas with $n$ variables generated from a distribution $\Distribution$ over $m$ clauses, where the mutual independence assumption over clauses is relaxed to $k$-wise independence. We use the term $k$-clause independence to refer to such distributions. We denote such families of distributions by $\DistFamily_k(n,m)$, where each $\Distribution\in\DistFamily_k(n,m)$ has identical marginals uniformly distributed over all possible OR-clauses and those marginals are only assumed to be $k$-wise independent in $\Distribution$. We would like to understand the following question for small values of $k$:
\begin{quote}
How does the satisfiability threshold $r_3$ of random $3$-SAT formulas behave under any $k$-clause independent distribution $\Distribution\in\DistFamily_k(n,m)$?
\end{quote}
As the distribution $\Distribution$ is not unique, there might be a large gap between lower and upper estimates of $r_3$. To this end, we formally define the \emph{lower satisfiability threshold} $\LST_k(n)$ as an upper bound on $m$, such that a random formula $\phi$ drawn from a distribution in $\DistFamily_k(n,m)$ with $m\le\LST_k(n)$ clauses has $\Prx{\phi \text{ is satisfiable}}\ge\frac{2}{3}$. Similarly, the \emph{upper satisfiability threshold} $\UST_k(n)$ is a lower bound on $m$, such that the random formula $\phi$ with $m\ge\UST_k(n)$ clauses has $\Prx{\phi \text{ is satisfiable}}\le\frac{1}{3}$. What kind of bounds on upper $\UST_k(n)$ and lower $\LST_k(n)$ thresholds should we expect?  

\paragraph{Reasonable expectations.}
The condition $\Distribution\in\DistFamily_k(n,m)$ only says something about configurations of at most $k$ clauses and does not put any other restrictions on the random formula $\phi\sim\Distribution$. As the degree of independence $k$ is a small constant, any argument that gives
bounds on $\LST_k$ or $\UST_k$ can only rely on statistics of at most a constant number of clauses. Hence, it is rather likely that  
bounds on $\LST_k$ and $\UST_k$ come together with efficient procedures of, respectively, finding a satisfying assignment for a random formula $\phi$, or certifying that $\phi$ is not satisfiable. Hence, given the prior work on random $3$-SAT for $\Distind\in\DistFamily_k(n,m)$, we get the following picture:
\begin{description}
    \item[Upper satisfiability threshold.] The best known result for refuting $3$-CNF formulas efficiently is due to Feige and Ofek~\cite{FeigeO07}, who show how to do it only for a fairly large number of clauses $m=c\cdot n^{3/2}$. Furthermore, for any smaller number of clauses $m=O(n^{3/2-\eps})$, a random $3$-CNF formula is likely to have only exponential in $n$ proof size for any unsatisfiability proof in resolution~\cite{BenSassonW01}.
    Hence, it is out of reach to aim for a better bound on $\UST_k(n)$ than $O(n^{3/2})$ while relying only on $k$-wise independence for some constant $k$. In fact, the best known positive result on \emph{efficiently computable proofs} of unsatisfiability in resolution is due to Beame et al.~\cite{BKTS02}, who show that an ordered DLL algorithm executed on a random $3$-SAT instance with $m=\Omega(n^2/\log n)$ clauses terminates in polynomial time.
     \item[Lower satisfiability threshold.] As the proofs for the lower bounds on $r_3$ often establish simple procedures that find satisfying assignments with high probability, it is still possible that $\LST_k(n)$ is of similar order $\Theta(n)$ as the lower bounds on $r_3$ for $\Distind$. Thus, the most ambitious result would be to show that $\LST_k(n)\le c_k\cdot n$ for constant $c_k$ that increases with $k$. A more modest goal is to aim for $\LST_k(n)=o(n)$ for a constant $k$, where $\LST_k(n)\to \Theta(n)$ as $k\to+\infty$. 
\end{description}    

\subsection{Our results}
\label{sec:results}
We obtain the following bounds on the lower and upper satisfiability thresholds $\LST_k(n)$ and $\UST_k(n)$ for various values of $k$.

\paragraph{Lower satisfiability thresholds.} We show (in  \textbf{Theorem~\ref{thm:lower-bound-k-wise-unsat}}) that $\LST_k(n)\ge \Omega(n^{1-1/k})$ for any $k\ge 2$. I.e., any $k$-clause independent random formula is satisfiable with high probability if it contains at most $O(n^{1-1/k})$ clauses. 
The argument is simple: for any $k$-clause independent distribution, we look at the 
$3$-uniform hypergraph that corresponds to the variables of a random formula $\phi$ produced according to this distribution, and argue that this graph does not have Berge-cycles, with high probability. We also provide an informal justification that this bound is asymptotically tight, i.e., that $\LST_k(n)=O(n^{1-1/k})$.
Specifically, we outline a plausible approach for constructing a $k$-clause independent distribution with $m=O(n^{1-1/k})$ clauses such that most of its formulas are unsatisfiable. Our approach is built upon existing constructions of dense hyper-graphs with large girth. It is interesting to note that, in both the proof of the $\LST_k(n)=\Omega(n^{1-1/k})$ result and the approach for showing that $\LST_k(n)=O(n^{1-1/k})$, we only need to consider variables and can completely ignore the distribution over the literals. 

\paragraph{Upper satisfiability thresholds.} We first consider small degrees of independence, i.e., $k\in \{2,3\}$. In both cases, we show that $\UST_k(n)=\Theta(n^3)$, meaning that one needs almost all possible clauses in a $3$-clause (as well as $2$-clause) independent formula to ensure that it is unsatisfiable (see \textbf{Theorem~\ref{thm:tight-result-2wise}} and \textbf{Theorem~\ref{thm:3wise_tightness}}). 
The most nontrivial part is to construct the distribution $\Distribution\in\DistFamily_3(n,m)$ with $m=\Theta(n^3)$ and $\Prx{\phi \text{ is satisfiable}}\ge\frac{2}{3}$. Our construction is based on ``$3$-XOR formulas'' (i.e., OR-clauses that have either one or three literals that are satisfied by a randomly planted truth assignment), which aligns well with the 
intuition developed in previous work~\cite{Feige02,FeigeO07}. The main technical difficulty is to ensure $k$-clause independence by adding a small fraction of unsatisfiable instances and checking all $3$-wise statistics.

Our most exciting and technically involved result (see \textbf{Theorem~\ref{thm:upper-bound-result-4wise}}) is our proof that $\UST_4(n)=O(n^2)$, i.e., a random formula $\phi\sim\Distribution$ with $m=O(n^2)$ clauses is unsatisfiable with large probability for any $4$-clause independent distribution $\Distribution\in \DistFamily_4(n,m)$. It is worth noting that such a bound is much harder to get under the $4$-wise independence assumption than in the case of a mutually independent distribution $\Distind$.
Indeed, Feige and Ofek~\cite{FeigeO07} describe a very simple refutation algorithm for $m=\Theta(n^2)$ that fixes a variable $x$ and considers all clauses containing $x$ or $\overline{x}$ (there will be $\Theta(n)$ such clauses in expectation). Then, after deleting $x$ (or $\overline{x}$), one can reduce the problem to the refutation of the respective random $2$-CNF sub-instance, which can be easily verified in polynomial time and has a low satisfiability threshold of $r_2=1$. This simple approach obviously fails for $4$-clause independent distributions. We instead construct a bipartite multigraph $G(\phi)$ between pairs of distinct literals on one side and all singleton literals on the other, in which every OR-clause in $\phi$ corresponds to three different edges. We then carefully examine the statistic $\kappa(\phi)$ that counts $K_{2,2}$ subgraphs in $G(\phi)$ for a random $\phi\sim\Distribution$. We find that the expected value of $\kappa(\phi)$ for random $\phi$ is only slightly larger than its absolute minimal value, while at the same time $\kappa(\phi)$ is significantly larger than its expectation when $\phi$ is satisfiable.
Our argument bears certain similarities with the argument in~\cite{FeigeO07}, which also looked at intersections of two literals between pairs of clauses but used the $3$-XOR principle and 
 a differently constructed non-bipartite graph.

\subsection{Roadmap}
The rest of the paper is structured as follows. We begin with preliminary definitions and notation in Section~\ref{sec:prelim}. Then, we warm up with our tight bounds on the upper satisfiability threshold $\UST_2(n)$ in Section~\ref{sec:LargeSat-2}. Our lower bound on $\UST_3(n)$ is proved in Section~\ref{sec:LargeSat-3} while our upper bound on $\UST_4(n)$ follows in Section~\ref{sec:LargeSat-4}. Section~\ref{sec:SmallUnsat-proof} is devoted to the study of the lower satisfiability threshold.

\section{Preliminaries}
\label{sec:prelim}
    Let $x_1$, $x_2$, $\cdots$, $x_n$ be $n$ boolean variables. A literal $\ell$ is a boolean variable or the negation of it. For convenience, we usually represent a literal as a variable-sign pair, i.e. $\ell = (x_i,s)$, where $s$ is the positive sign $\Pos$ if the literal is a boolean variable and the negative sign $\Neg$ if it is its negation. We define $\SignSet(n)=\set{\Pos,\Neg}^n$. 
An instance of the \emph{Satisfiability} problem in conjunctive normal form (or \emph{SAT instance}, for short) is a boolean formula over a subset of the $n$ variables which is a conjunction of disjunctive clauses, each clause containing literals with different variables. In a $3$-SAT instance, every clause has exactly $3$ literals. 
Each SAT instance $\ClauseSet$ can be described as a multiset of clauses. We denote the size of an instance $\ClauseSet$ as $m=|\ClauseSet|$. 

Given the number of variables $n$, let $X(n)$ be the set of variables. Let $\TripletVarSet(n)$ be the set of all unordered triplets of variables $\TripletVarSet(n)=\set{(x_i,x_j,x_k) ~|~ x_i,x_j,x_k\in X(n), i<j<k}$ and  $\ClauseTypeSet(n)$ be the set of all possible clauses, i.e., all triplets with literals of different variables. 
We will usually omit the dependency of $X, \TripletVarSet,$ and $\ClauseTypeSet$ on $n$, when the value of $n$ is clear from the context. 
To refer to the variable (the set of variables) or the sign (the set of signs) of a literal (a clause $c\in\ClauseSet$), we define operators $\GetVariable(\cdot)$ and $\GetSign(\cdot)$, so that if, e.g., $c = ((x_1,\Pos),(x_2,\Neg),(x_3,\Pos))$, then $\GetVariable(c)=(x_1,x_2,x_3)$ and $\GetSign(c)=(\Pos, \Neg, \Pos)$. 

A \emph{truth assignment} is a vector $\sigma \in \set{0, 1}^n$, where $1$ or $0$ at the $\sigma_i$ coordinate corresponds to the ``true'' or ``false'' value of $x_i$, respectively. A truth assignment $\sigma$ is \emph{satisfying} for an instance $\ClauseSet$ when each clause $c\in\ClauseSet$ has at least one true literal under $\sigma$. An instance $\ClauseSet$ is \textit{satisfiable} if there exists a satisfying assignment; otherwise, $\ClauseSet$ is \textit{unsatisfiable}.  

The random $3$-SAT model assumes that the instance is drawn from a probability distribution over all possible $3$-SAT instances with $m$ clauses and $n$ variables. The main object of study in such a model is the probability of such a random instance being satisfiable as a function of $n$ and $m$.

\subsection{Random $3$-SAT without mutual independence}
In the standard random $3$-SAT model, each clause is drawn uniformly at random among all clauses in $\ClauseTypeSet(n)$, independently from the other clauses. A constructive definition is given by Model~\ref{model:independent-sat}. We denote the distribution over instances generated by Model~\ref{model:independent-sat} as $\Distribution_{\text{Ind.}}$. 

\modelname
\begin{algorithm}
{\bf Input:} Integers $n\geq 3$ and $m\geq 1$.

{\bf Output:} A $3$-SAT instance with $n$ variables and $m$ clauses.

\caption{Selects a $3$-SAT instance uniformly at random from all possible instances}\label{model:independent-sat}

\begin{algorithmic}[1]
    \STATE $C\gets \emptyset$;
    \FOR{$l \gets 1, ..., m$}
        \STATE pick a literal triplet $(\ell_1,\ell_2,\ell_3)$ uniformly at random from $\ClauseTypeSet(n)$;
        \STATE $c\gets (\ell_1,\ell_2,\ell_3)$;
        \STATE $C\gets C\cup \{c\}$;
    \ENDFOR
    \RETURN $C$
\end{algorithmic}
    
\end{algorithm}

Our main focus will be on $k$-wise relaxations of the independent distribution.
\begin{definition}
[$k$-clause independent random SAT]
A distribution $\Distribution$ for selecting a random SAT instance is $k$-clause independent if the set of any fixed $k$ clauses $S\subseteq \ClauseSet$ is distributed uniformly at random over all $k$-tuples of possible clauses, i.e., 
    \begin{align*}
    \Prlong[\ClauseSet\sim\Distribution]{c_i=t_i, \forall i\in[k]}= \Prlong[\ClauseSet\sim\Distribution_{\text{Ind.}}]{c_i=t_i, \forall i\in[k]}=\prod_{i\in[k]} \Prx{c_i=t_i}.
    \end{align*}
    for every $c_1,\ldots,c_k\in \ClauseSet$ and $t_1,\ldots,t_k\in \ClauseTypeSet$.
    We denote the family of all $k$-clause independent distributions with $m$ clauses over $n$ variables as $\DistFamily_k(n,m)$.
\end{definition}
We remark that, in contrast to the random $3$-SAT model (Model~\ref{model:independent-sat}), which defines a single distribution for given $n$ and $m$, the family $\DistFamily_k(n,m)$ contains many different distributions.

By definition, the probability of any event $A_S$ that depends only on a subset $S$ of at most $k$ clauses in the $k$-clause independent distribution $\Distribution$ is the same as for $\Distribution_{\text{Ind.}}$. I.e., the expectations of the indicator function $\ind{A_S}$ are the same for $\Distribution$ and $\Distribution_{\text{Ind.}}$. Hence, by linearity of expectation, any statistic that involves only $k'\le k$ clauses must be the same for $\Distribution$ and  $\Distribution_{\text{Ind.}}$, i.e., 
\begin{align}
    \label{eq:exp-between-k-wise-and-ind.}
    \Exlong[\ClauseSet \sim \Distribution]{\sum_{S \subseteq \ClauseSet,\atop |S|=k'}\ind{A_{S}}}=\Exlong[\ClauseSet\sim \Distribution_{\text{Ind.}}]{\sum_{S\subseteq \ClauseSet,\atop |S|=k'}\ind{A_{S}}}.
\end{align}

We would like to understand what are the largest/smallest possible number of clauses for a random $3$-SAT formula to be almost surely satisfiable/unsatisfiable under any $k$-clause independent distribution. Using $SAT(\ClauseSet)$ to denote the event that the SAT instance $\ClauseSet$ is satisfiable, we define the following satisfiability ``thresholds''.

\begin{definition}[\LargestSizeOfSatisfiableInstance]
\label{def:LargestSat}
The upper satisfiability threshold function $\UST_k(n)$ is defined as follows. For integer $n\geq 3$, $\UST_k(n)$ is the minimum integer $m$ such that 
$$\Prlong[\ClauseSet \sim \Distribution]{SAT(\ClauseSet)} \le 1/3,~~\forall \Distribution \in \DistFamily_k(n,m).$$
\end{definition}

\begin{definition}[\SmallestSizeOfUnsatisfiableInstance]
\label{def:SmallestUnsat}
The lower satisfiability threshold function $\LST_k(n)$ is defined as follows. For integer $n\geq 3$, $\LST_k(n)$ is the maximum integer $m$ such that 
$$\Prlong[\ClauseSet \sim \Distribution]{\neg SAT(\ClauseSet)} \le 1/3,~~\forall \Distribution \in \DistFamily_k(n,m).$$
\end{definition}
Extending the line of research on the standard random $3$-SAT model, we would like to have as tight estimates of $\UST_k(n)$ and $\LST_k(n)$ as possible.

\section{Tight bounds for the upper satisfiability threshold $\UST_2(n)$}
\label{sec:LargeSat-2}
        We begin with a technical warm up and prove asymptotically tight bounds on the upper satisfiability threshold of $2$-clause independent random $3$-SAT.
\begin{theorem}
    $\UST_2(n)=\Theta(n^3)$.
    \label{thm:tight-result-2wise}
\end{theorem}
Theorem~\ref{thm:tight-result-2wise} follows by the next two lemmas. Lemma~\ref{lem:upper-bound-pairwise-sat} provides an upper bound on $\UST_2(n)$. In the proof, we introduce the technique that we will use later in Section~\ref{sec:LargeSat-4} to get a much more involved upper bound on $\UST_4(n)$. 
\begin{lemma}
    For any $\Distribution\in\DistFamily_2(n,m)$ with $m\ge 56\binom{n}{3}$, $\Prx[\ClauseSet\sim\Distribution]{SAT(\ClauseSet)}\le{56\binom{n}{3}}/{m}$.
    \label{lem:upper-bound-pairwise-sat}
\end{lemma}

\begin{proof}
Let $\xi(\ClauseSet) \eqdef \sum_{c_i,c_j\in \ClauseSet}\ind{c_i=c_j}$ be the number of identical clause pairs in the instance $\ClauseSet$. On the one hand, Equation~\eqref{eq:exp-between-k-wise-and-ind.} implies that the expectation of $\xi(\ClauseSet)$ is the same when $\ClauseSet$ is drawn from a $2$-clause independent distribution and $\distind$. On the other hand, if an instance $\ClauseSet$ has a satisfying assignment, at least $1/8$ of possible clauses from $\ClauseTypeSet$ must not appear in $\ClauseSet$, which means that the value of $\xi(\ClauseSet)$ is significantly higher than its expectation. These two observations will allow us to get the desired bound on the probability $\Prx[\ClauseSet\sim\dist]{SAT(\ClauseSet)}$. 

Let $\Distribution\in \DistFamily_2(n,m)$ and define $p\eqdef\Prx[\ClauseSet\sim\dist]{SAT(\ClauseSet)}$. We shall derive two lower bounds on the value the random variable $\xi(\ClauseSet)$ can get: an unconditional lower bound $\LB[\xi]$ (not far from $\Ex[\ClauseSet \sim \Distribution]{\xi(\ClauseSet)}$) and a lower bound $\LB[\xi~|~SAT]$ on $\xi(\ClauseSet)$ when $\ClauseSet$ is satisfiable (this will be significantly larger than $\Ex[\ClauseSet \sim \Distribution]{\xi(\ClauseSet)}$). We note that
\begin{align}\nonumber
    \Ex[\ClauseSet \sim \Distribution]{\xi(\ClauseSet)} &= \Prx{SAT(\ClauseSet)}\cdot \Ex{\xi(\ClauseSet)~|~SAT(\ClauseSet)} + \Prx{\neg SAT(\ClauseSet)}\cdot\Ex{\xi(\ClauseSet)~|~\neg SAT(\ClauseSet)}\\\label{inq:main-2wise}
    &\ge
    p\cdot \LB[\xi~|~SAT] + (1-p) \cdot \LB[\xi].\quad
\end{align}

We next derive $\Ex[\ClauseSet \sim \Distribution]{\xi(\ClauseSet)}$, $\LowerBound{\xi}$, and $\LowerBound{\xi~|~SAT}$. We denote as $M\eqdef|\ClauseTypeSet|=8\binom{n}{3}$ the number of possible clauses and $\lambda\eqdef m/M > 1$.
First, we observe that 
\begin{equation}
    \label{eq:expect_xi_2wise}
    \Ex[\ClauseSet \sim \Distribution]{\xi(\ClauseSet)}=\sum_{c_i,c_j\in \ClauseSet}\Ex{\ind{c_i=c_j}}
    =\frac{m^2-m}{2}\cdot \Prx{c_i=c_j}=
    \frac{m^2-m}{2\cdot M}=m\cdot\frac{\lambda\cdot M-1}{2\cdot M}.
\end{equation}
To derive the lower bounds $\LowerBound{\xi}$ and 
$\LowerBound{\xi~|~SAT}$, we observe that any given clause type $t\in\ClauseTypeSet$ contributes $\binom{d_t}{2}$ pairs
to $\xi(\ClauseSet)$, where $d_t\eqdef\sum_{c\in\ClauseSet}\ind{c=t}$ is the number of type $t$ clauses in $\ClauseSet$. I.e.,
\begin{equation}
    \label{eq:optimization_xi}
    \xi(\ClauseSet)=\sum\limits_{t\in\ClauseTypeSet}
    \frac{d_t^2-d_t}{2},\quad\quad\quad\text{where }
    \sum\limits_{t\in\ClauseTypeSet}d_t=m.
\end{equation}
The minimum of $\xi(\ClauseSet)=\frac{1}{2}\sum d_t^2-\frac{m}{2}$ under the constraint $\sum d_t=m$ is achieved when all the $d_t$ variables are equal, i.e., equal to $\frac{m}{|\ClauseTypeSet|}=\lambda$. Thus, we get the following lower bound $\LB[\xi]$ on $\xi(\ClauseSet)$:
\begin{equation}
    \label{eq:general_lb_xi}
    \xi(\ClauseSet)\ge\LowerBound{\xi}\eqdef\frac{1}{2} \sum_{t\in \ClauseTypeSet}\lambda^2 -\frac{m}{2}=
    m\cdot\frac{\lambda-1}{2}.    
\end{equation}

To derive the lower bound $\LowerBound{\xi~|~SAT}$ on $\xi(\ClauseSet)$ for a satisfiable formula $\ClauseSet$, we note that at least a $\frac{1}{8}$-fraction of the clause types are not present in $\ClauseSet$. I.e., at least $\frac{M}{8}$ of the $d_t$ variables have value equal to $0$ in~\eqref{eq:optimization_xi}. Similarly to \eqref{eq:general_lb_xi}, the minimum value of $\xi(\ClauseSet)$ is achieved when all the remaining $\frac{7M}{8}$ $d_t$ variables are equal to each other, getting the value $\frac{8\cdot m}{7\cdot M}$. That is,
\begin{equation}
\label{eq:satisfiable_lb_xi}
\xi(\ClauseSet)\ge\LowerBound{\xi~|~SAT}\eqdef \frac{7\cdot M}{8}\cdot\left(\frac{8\lambda}{7}\right)^2-\frac{m}{2}= m\cdot\frac{\frac{8}{7}\lambda-1}{2}.    
\end{equation}
We plug the bounds \eqref{eq:expect_xi_2wise},\eqref{eq:general_lb_xi}, and \eqref{eq:satisfiable_lb_xi} into \eqref{inq:main-2wise} to get 
\[
    m\cdot\frac{\lambda\cdot M-1}{2\cdot M}\ge
    p\cdot m\cdot\frac{\frac{8}{7}\lambda-1}{2}+(1-p)\cdot m\cdot\frac{\lambda-1}{2}.
\]
After simple algebraic transformation, this is equivalent to the inequality $1-\frac{1}{M}\ge\frac{p}{7}\lambda$, which implies that $p\le\frac{7\cdot M}{m}=56\binom{n}{3}/m$.
\end{proof}

To lower-bound $\UST_2(n)$, we use a specific $2$-clause independent distribution, which is depicted as Model~\ref{model:pair-wise-construction}. The next lemma proves the correctness of our construction.
\begin{lemma}
    Model \ref{model:pair-wise-construction} defines a $2$-clause independent probability distribution that generates 3-SAT instances of size $\Omega(n^3)$ that are satisfiable with probability at least $1-O(n^{-3})$.
    \label{lem:lower-bound-pairwise-sat}
\end{lemma}

\modelname
\begin{algorithm}
{\bf Input:} Integer $n\geq 3$.

{\bf Output:} A $3$-SAT instance $C$ with $n$ variables and $m=\binom{n}{3}$ clauses.

\caption{Selects a $3$-SAT instance according to a $2$-clause independent distribution} \label{model:pair-wise-construction}

\begin{itemize}
    \item With probability $1-\binom{n}{3}^{-1}$, construct a satisfiable instance $\ClauseSet$ 
    \begin{enumerate}
        \item Match $m=\binom{n}{3}$ clauses to all different $\TripletVarSet(n)$ variable triplets uniformly at random;
        \item Pick a random truth assignment $\sigma\sim\uni[\{0,1\}^n]$;
        \item For each clause $c\in\ClauseSet$ matched to the variable triplet $(x_i,x_j,x_k)\in T(n)$
        \begin{itemize}
            \item Pick a random sign triplet 
            $(s_1,s_2,s_3)$ of the same parity with $(\sigma_i,\sigma_j,\sigma_k)$\\ (i.e., $\ind{s_1=\Neg}+\ind{s_2=\Neg}+\ind{s_3=\Neg}+\sigma_i+\sigma_j+\sigma_k=0\mod 2$);
            \item Let clause $c\gets ((x_i,s_1),(x_j,s_2),(x_k,s_3))$;
        \end{itemize} 
    \end{enumerate}
    \item With probability $\binom{n}{3}^{-1}$, sample $\ClauseSet$ from distribution $\univar$ defined as follows:
    \begin{enumerate}
        \item Pick a random \textbf{single} variable triplet $(x_i,x_j,x_k)\sim\uni[T(n)]$;
        \item For each clause $c\in\ClauseSet$
        \begin{itemize}
            \item Pick a random sign triplet 
            $(s_1,s_2,s_3)\sim\uni[\{\Pos,\Neg\}^3]$;
            \item Let clause $c\gets ((x_i,s_1),(x_j,s_2),(x_k,s_3))$;
        \end{itemize}
    \end{enumerate}
\end{itemize}
\end{algorithm}

\begin{proof}
By its definition, with probability $1-\binom{n}{3}^{-1}$, Model~\ref{model:pair-wise-construction} returns a $3$-SAT instance that has a planted satisfying assignment $\sigma$.  Indeed, the condition
$\ind{s_1=\Neg}+\ind{s_2=\Neg}+\ind{s_3=\Neg}+\sigma_i+\sigma_j+\sigma_k=0\mod 2$ ensures that the assignment $x_i=\sigma_i, x_j=\sigma_j, x_k=\sigma_k$ satisfies clause $c=((x_i,s_1),(x_j,s_2),(x_k,s_3))$. Hence, using $\dist$ to denote the distribution according to which Model~\ref{model:pair-wise-construction} selects the $3$-SAT instance, we have $\Prx[\ClauseSet\sim\dist]{SAT(\ClauseSet)}\geq 1-\binom{n}{3}^{-1}$.

We are only left to verify that the distribution $\dist$ is $2$-clause independent. We first consider the distribution $\GetVariable(c)$ of variable triplets for $c\in\ClauseSet$ with $\ClauseSet\sim\dist$ and notice that it is pair-wise independent. 
Indeed, the probability of drawing a pair of identical variable triplets in a $2$-clause independent distribution is $\binom{n}{3}^{-1}$, while the probability of drawing two different variable triplets is $1-\binom{n}{3}^{-1}$. We exactly match these probabilities in our construction, as we pick a pair of non-identical variable triplets uniformly at random with probability $1-\binom{n}{3}^{-1}$ and with probability $\binom{n}{3}^{-1}$ generate a pair of clauses with identical set of variables.

For the distribution of sign triplets, we first note that distribution $\univar$ is a uniform pair-wise independent distribution for each single triplet of variables $(x_i,x_j,x_k)$. 
Second, we notice that when variable triplets are different in a pair of clauses $c_1,c_2$, then each of these clauses has at least one unique variable that does not appear in another clause (i.e., $x_i\in\GetVariable(c_1)$ and $x_i\notin\GetVariable(c_2)$, while  
$x_\ell\in\GetVariable(c_2)$ and $x_\ell\notin\GetVariable(c_1)$). 
As we pick in our construction the satisfying assignment $\sigma$ uniformly at random, the parity condition for the sign triplet in clause $c_2$ does not affect the sign of $x_\ell$ variable for any fixed sign triplet of $c_1$. Hence, we get all combinations of sign triplets in $c_1$ and $c_2$ (with a fixed set of variables) with the same probability. Therefore, our construction is $2$-clause independent.
\end{proof}

\section{A tight lower bound for $\UST_3(n)$}
    \label{sec:LargeSat-3}
        Clearly, Lemma~\ref{lem:upper-bound-pairwise-sat} also provides an upper bound on the upper satisfiability threshold $\UST_3(n)$. The proof of the next statement follows by presenting a matching lower bound through Model~\ref{model:3-wise-construction}.

\begin{theorem}
    \label{thm:3wise_tightness}
 $\UST_3(n)=\Theta(n^3)$.
\end{theorem}
\begin{proof}
We give an explicit construction (see Model~\ref{model:3-wise-construction}) of a $3$-clause independent distribution $\dist\in\DistFamily_3(n,m)$ with $n$ variables and (an even number of) $m\le\frac{1}{3}\cdot\binom{n}{3}$ clauses, such that $\Prx[\ClauseSet\sim\dist]{SAT(\ClauseSet)}\ge2/3-O(n^{-6})$. Our construction in Model~\ref{model:3-wise-construction} follows the same pattern as in Model~\ref{model:pair-wise-construction}, but uses an additional step and minor modifications to ensure $3$-clause independence.\footnote{The first and third block of Model~\ref{model:3-wise-construction} correspond to the two blocks of Model~\ref{model:pair-wise-construction}. The only difference in the first block is that the matching of $C$ is not to all the variable triplets in $T(n)$ but only to $m$ of them. }

\begin{algorithm}[ht]
{\bf Input:} Integer $n\geq 3$ and even integer  $m\le \frac{1}{3}\cdot \binom{n}{3}$.

\vspace{0.3em}

{\bf Output:} A $3$-SAT instance $C$ with $n$ variables and $m$ clauses.

\caption{Selects a $3$-SAT instance according to a $3$-clause independent distribution} \label{model:3-wise-construction}

Set $p\eqdef(m-1)\cdot\left(\binom{n}{3}^{-1}-\frac{1}{3}\binom{n}{3}^{-2}\right)$
and $q\eqdef\binom{n}{3}^{-2}$;

\vspace{0.3em}

\begin{itemize}
    \item With probability $1-p-q$, construct a satisfiable instance $\ClauseSet$
    \begin{enumerate}
        \item Match $\ClauseSet$ clauses to $m$ different variable triplets in $T(n)$ uniformly at random;
        \item Pick a random truth assignment $\sigma\sim\uni[\{0,1\}^n]$;
        \item For each clause $c\in\ClauseSet$ matched to the variable triplet $(x_i,x_j,x_k)\in T(n)$
        \begin{itemize}
            \item Pick a random sign triplet 
            $(s_1,s_2,s_3)$ of the same parity with $(\sigma_i,\sigma_j,\sigma_k)$\\ (i.e., $\ind{s_1=\Neg}+\ind{s_2=\Neg}+\ind{s_3=\Neg}+\sigma_i+\sigma_j+\sigma_k=0\mod 2$);
            \item Let clause $c\gets ((x_i,s_1),(x_j,s_2),(x_k,s_3))$;
        \end{itemize} 
    \end{enumerate}
    \item With probability $p$, construct an instance $\ClauseSet$ with $m/2$ different variable triplets
    \begin{enumerate}
        \item Uniformly at random match $\ClauseSet$ to $\frac{m}{2}$ different variable triplets in $T(n)$\\ 
        (exactly $2$ clauses in $\ClauseSet$ per one variable triplet);
        \item For each clause $c\in\ClauseSet$ assigned to the variable triplet $(x_i,x_j,x_k)\in T(n)$
        \begin{itemize}
            \item Pick a random sign triplet 
            $(s_1,s_2,s_3)\sim\uni[\{\Pos,\Neg\}^3]$;
            \item Let clause $c\gets ((x_i,s_1),(x_j,s_2),(x_k,s_3))$;
        \end{itemize}
    \end{enumerate}
    \item With probability $q$, sample $\ClauseSet$ from distribution $\ClauseSet\sim\univar$ as follows:
        \begin{enumerate}
        \item Pick a random single variable triplet $(x_i,x_j,x_k)\sim\uni[T(n)]$;
        \item For each clause $c\in\ClauseSet$
        \begin{itemize}
            \item Pick a random sign triplet 
            $(s_1,s_2,s_3)\sim\uni[\{\Pos,\Neg\}^3]$;
            \item Let clause $c\gets ((x_i,s_1),(x_j,s_2),(x_k,s_3))$;
        \end{itemize}
    \end{enumerate}
\end{itemize}
\end{algorithm}
Now, we need to differentiate between three types of situations for a triplet of clauses: type $I$, when all three clauses have pair-wise distinct sets of variables; type $II$,  when exactly one pair of clauses share the same variable triplet; type $III$, when all three clauses have the same set of variables. Notice that 
\[
\Prlong[\distind]{I}=1-3\cdot\binom{n}{3}^{-1},\quad\quad
\Prlong[\distind]{II}=3\cdot\binom{n}{3}^{-1}-\binom{n}{3}^{-2},\quad\quad
\Prlong[\distind]{III}=\binom{n}{3}^{-2}.
\]
The distribution $\dist$ used by Model~\ref{model:3-wise-construction} consists of three distributions: $\dist_1$ (sampled with probability $1-p-q$), $\dist_2$ (sampled with probability $p$), and $\dist_3=\univar$ (sampled with probability $q$). The parameters $p$ and $q$ are chosen so that 
\[
\Prlong[\ClauseSet\sim\dist]{II}=p\cdot\Prlong[\dist_2]{II}=p\cdot\frac{3}{m-1}=\Prlong[\distind]{II};\quad\quad
\Prlong[\ClauseSet\sim\dist]{III}=q=\Prlong[\distind]{III}.
\]
I.e., we perfectly match the probabilities of the events $I,II,$ and $III$ for $3$-SAT instances produced by distributions $\dist$ and $\distind$. Note that, similarly to Model~\ref{model:pair-wise-construction}, the generated $3$-SAT instance is always satisfiable when $\ClauseSet\sim\dist_1$. The numbers $p$ and $q$ in the definition of Model~\ref{model:pair-wise-construction} are chosen so that $p+q\le \frac{1}{3}+O(n^{-6})$ for $m\le\frac{1}{3}\cdot \binom{n}{3}$. Hence, $\Prx[\ClauseSet\sim\dist]{SAT(\ClauseSet)}\geq 2/3-O(n^{-6})$.

It remains to verify that distribution $\dist$ is indeed $3$-clause independent. Let us fix any three clauses $c_1,c_2,c_3\in\ClauseSet$ with $\ClauseSet\sim\dist$.
We shall separately study the distributions of variable triplets and sign triples of $(c_1,c_2,c_3)$. To this end, we define $\GetVariable(c_i)\eqdef\hTi\in T(n)$ and $\GetSign(c_i)\eqdef\hrhoi\in\{\Pos,\Neg\}^3$ for each $i\in[3]$. 
Let $T_1,T_2,T_3\in T(n)$ and 
$\rho_1,\rho_2,\rho_3\in\{\Pos,\Neg\}^3$ be respectively any three fixed variable triplets and any three fixed sign triplets.
Since we perfectly match the probabilities of types $I,II,III$ for variable triples with the respective probabilities in $\distind$, and we pick variable triplets uniformly at random in each $\dist_1,\dist_2,$ and $\dist_3$, we have
\[
\Prlong[\ClauseSet\sim\dist]{\hTi=T_i,~i\in[3]}=
\Prlong[\ClauseSet\sim\distind]{T_i,~i\in[3]}.
\]
Thus, the respective distribution of variable triples generated by distribution $\dist$
is $3$-wise independent. For the distribution of sign triplets, we note that we sample signs independently from the variable triplets and uniformly over $\{\Pos,\Neg\}^3$ in
either of $\dist_2$ and $\dist_3$. Hence, it only remains to show that the distribution of sign triplets in $\dist_1$ is uniform $3$-wise independent for any feasible triplets of variables $(T_1,T_2,T_3)$ such that $T_1\ne T_2, T_2\ne T_3$, and $T_1\ne T_3$. We prove this by considering two cases. First, when variable triples $T_1, T_2$, $T_3$ have at least $5$ different variables among them.
\begin{claim}
    \label{cl:3wise_union5}
    When $|T_1\cup T_2\cup T_3|\ge 5$ and $T_1\ne T_2$, $T_2\ne T_3$, $T_1\ne T_3$, we have 
    \[
    \Prlong[\ClauseSet\sim\dist_1]{\hrhoi=\rho_i,~i\in[3]~\Big\vert~\hTi=T_i,~i\in[3]}=\Prlong[\ClauseSet\sim\distind]{\rho_i,~i\in[3]}.
    \]
\end{claim}
\begin{proof}
In this case, at least one of the variable triplets $T_1$, $T_2$, or $T_3$ must have a variable that is not present in the other two triplets. Indeed, if we count appearances of the variables in $T_1, T_2, T_3$, we get $3\cdot 3$. On the other hand, if each variable in $T_1\cup T_2\cup T_3$ appears at least twice, our count must be at least $2\cdot 5 > 9$ -- a contradiction. 
Hence, we may assume, without loss of generality, that variable triplet $T_3$ has a unique variable
$z\in T_3$ such that $z\notin T_1, T_2$.

As we choose the satisfying assignment $\sigma$ in $\dist_1$ uniformly at random, $\sigma(z)\in\{0,1\}$ equally likely for any fixed assignment of variables in $T_1\cup T_2$. Therefore, the choice of sign triplet for $T_3$ for any choice of sign triplets in $T_1$ and $T_2$ is uniform over $\{\Pos,\Neg\}^3$. This means that the choice of sign triplet for $c_3$ in $\dist_1$ is uniform and independent of the clauses $c_1$ and $c_2$. By the same reasoning as in the proof of Lemma~\ref{lem:lower-bound-pairwise-sat} for Model~\ref{model:pair-wise-construction}, the sign triplets in $c_1$ and $c_2$ are uniform over $\{\Pos,\Neg\}^3$ and pairwise independent in $\dist_1$.   
\end{proof}

Next, we consider the case when $T_1,$ $T_2,$ and $T_3$ have exactly $4$ variables among them.
\begin{claim}
    \label{cl:3wise_union4}
    When $|T_1\cup T_2\cup T_3|=4$ and $T_1\ne T_2$, $T_2\ne T_3$, $T_1\ne T_3$, we have 
    \[
    \Prlong[\ClauseSet\sim\dist_1]{\hrhoi=\rho_i,~i\in[3]~\Big\vert~\hTi=T_i,~i\in[3]}=\Prlong[\ClauseSet\sim\distind]{\rho_i,~i\in[3]}.
    \]
\end{claim}

\begin{proof}
Let $T_1\cup T_2\cup T_3=\{x_1,x_2,x_3,x_4\}$.
We shall prove a stronger statement that four sign triplets $\hrhoi\in\{\Pos,\Neg\}^3$ for $i\in[4]$ with $\ClauseSet\sim\dist_1$ are $4$-wise independent when their respective variable triplets $T_1,T_2,T_3,T_4$ are all four different subsets of $\{x_1,x_2,x_3,x_4\}$.
To this end, consider a system of $4$ equations with $4$ variables over the finite field $\field$
\begin{equation}
\label{eq:xor4equations}
    \begin{cases}
    x_1\oplus x_2\oplus x_3  = v_1 & \quad\text{corresponds to }\hTi[1],\hrhoi[1]\\
    x_1\oplus x_2\oplus x_4  = v_2 & \quad\text{corresponds to }\hTi[2],\hrhoi[2]\\
    x_1\oplus x_3\oplus x_4  = v_3 & \quad\text{corresponds to }\hTi[3],\hrhoi[3]\\
    x_2\oplus x_3\oplus x_4  = v_4 & \quad\text{corresponds to }\hTi[4],\hrhoi[4].
  \end{cases}
\end{equation}
System~\eqref{eq:xor4equations} always has a unique solution for any $v_1,v_2,v_3,v_4\in\{0,1\}$. I.e., for any choice of sign triplets $(\rho_i)_{i\in[4]}$, there is a unique truth assignment $\sigma$ restricted to variables $x_1,x_2,x_3,x_4$ that satisfies the parity constraint from our construction of $\dist_1$ (for each $\rho_i = (s_1^i,s_2^i,s_3^i)$, $i\in[4]$, we set $v_i=\ind{s_1^i=\Neg}\oplus\ind{s_2^i=\Neg}\oplus\ind{s_3^i=\Neg}$ and the restriction of the truth assignment $\sigma$ to variables $x_1,x_2,x_3,x_4$ is the solution to the corresponding system of linear equations~\eqref{eq:xor4equations}). As we choose the satisfying assignment $\sigma$ uniformly at random in $\dist_1$, the sign triplets $(\hrhoi)_{i\in[4]}$ are independent and uniformly distributed over $\{\Pos,\Neg\}^3$.
\end{proof}

The above two cases in Claims~\ref{cl:3wise_union5} and \ref{cl:3wise_union4} cover all the possibilities for the variable triplets generated in $\dist_1$, as we always have $\hTi\ne\hTi[j]$ for different clauses in $\dist_1$, which concludes the proof of 
$3$-clause independence of distribution $\dist$. Since our construction has size $m=\Omega(n^3)$, and given that $\UB_3(n)\le\UB_2(n)=O(n^3)$, Theorem~\ref{thm:3wise_tightness} follows.  
\end{proof}

\section{An upper bound for $\UST_4(n)$}
    \label{sec:LargeSat-4}
        Our next result is rather surprising as it indicates that $4$-wise independence allows for a steep decrease in the upper satisfiability threshold compared to $2$- and $3$-wise independence.

\begin{theorem}\label{thm:upper-bound-result-4wise}
$\UST_4(n)=O(n^2)$.
\end{theorem}

\begin{proof}
We will prove the theorem by showing that for any positive integers $n$ and $m$ and any $4$-clause independent probability distribution $\Distribution\in\DistFamily_4(n,m)$, it holds $\Prx[\ClauseSet\sim\dist]{SAT(\ClauseSet)}\leq O\left(\max\left\{\frac{n^2}{m},\frac{1}{\sqrt{n}}\right\}\right)$. The claim is obvious for $n<10$. We will assume that $n\geq 10$ and $m\leq \sqrt{10}\cdot n^{5/2}$, and will show that $\Prx[\ClauseSet\sim\dist]{SAT(\ClauseSet)}\leq \frac{4288\cdot n^2}{m}$ for every $4$-clause independent probability distribution $\Distribution\in\DistFamily_4(n,m)$. Note that for $m>\sqrt{10}\cdot n^{5/2}$, the probability bound of $O\left(\frac{1}{\sqrt{n}}\right)$ follows by selecting uniformly at random a subset of $\sqrt{10}\cdot n^{5/2}$ clauses. 

We will use a graph representation of $3$-SAT instances defined as follows. Given a $3$-SAT instance $\ClauseSet$ consisting of $m$ clauses over $n$ variables, the bipartite multi-graph $G(\ClauseSet)=(L\cup R,E)$ has a node corresponding to each (unordered) pair of literals $\{\ell_1,\ell_2\}$ from different variables at the left node side $L$ and a node corresponding to each literal $\ell$ at the right node side $R$. Hence, $|L|=4\binom{n}{2}$ and $|R|=2n$. For every clause $c=(\ell_1,\ell_2,\ell_3)$ of $\ClauseSet$, $G(\ClauseSet)$ has the three edges between the node corresponding to the pair of literals $(\ell_i,\ell_j)$ and the node corresponding to literal $\ell_{6-i-j}$ for $(i,j)\in \{(1,2),(1,3),(2,3)\}$.

The main proof idea of Theorem~\ref{thm:upper-bound-result-4wise} is to analyse the statistic 
$\kappa(\ClauseSet)$ defined as the number of distinct $K_{2,2}$ subgraphs in graph $G(\ClauseSet)$. Namely, we first derive an upper bound on the expectation $\Ex[\ClauseSet\sim \Distribution]{\kappa(\ClauseSet)}$ (see 
Lemma~\ref{lem:4-wise-upper-bound-argument}) when $\ClauseSet$ is drawn from the $4$-clause independent probability distribution $\Distribution\in \DistFamily_4(n,m)$. 
We then give two lower bounds on the value of the random variable $\kappa(\ClauseSet)$ by considering an underlying simple subgraph of $G(\ClauseSet)$. Note that the underlying simple subgraph corresponds to a smaller instance $\tilC$, in which we remove all repeated clauses. As we show in Lemma~\ref{lem:tilm-vs-m}, this does not significantly reduce the size of the instance. Our first lower bound (Lemma~\ref{lem:lower-bound-4-wise-any}) on $\kappa(\tilC)$ holds for any instance $\tilC$ and is very close to the upper bound on the expectation of $\kappa(\ClauseSet)$. On the other hand, when $\tilC$ is satisfiable, we manage to give a significantly stronger lower bound on $\kappa(\tilC)$ in Lemma~\ref{lem:lower-bound-4-wise-sat}. We conclude the proof of 
Theorem~\ref{thm:upper-bound-result-4wise} by relating all these upper and lower bounds with $\Ex[\ClauseSet\sim \Distribution]{\kappa(\ClauseSet)}$ in 
Lemma~\ref{lem:lower-bound-for-kappa}.


\paragraph{Upper-bounding $\Ex[\ClauseSet\sim \Distribution]{\kappa(\ClauseSet)}$.} 
A $K_{2,2}$ subgraph in graph $G(\ClauseSet)$ is formed by edges corresponding to four clauses. In the next simple Lemma~\ref{lem:structure}, we describe necessary and sufficient conditions for a set of four clauses to create one or two $K_{2,2}$ subgraphs in $G(\ClauseSet)$. The expected number of $K_{2,2}$ subgraphs in $G(\ClauseSet)$ can be calculated by linearity of expectation: we simply estimate the probability that a set of $4$ clauses corresponds to $K_{2,2}$ in $G(\ClauseSet)$ for any $4$-wise independent distribution $\Distribution\in \DistFamily_4(n,m)$, which is the same as for $\Distribution_{\text{Ind.}}$ distribution. 
We apply Lemma~\ref{lem:structure} to get a lower bound (a tight estimate of) on the expected number of $K_{2,2}$ subgraphs in Lemma~\ref{lem:4-wise-upper-bound-argument}.

\begin{lemma}\label{lem:structure}
Any set of four clauses in $\ClauseSet$ satisfies one of the following $3$ cases. 
    \begin{enumerate}
    \item Form exactly one $K_{2,2}$ subgraph in $G(\ClauseSet)$ and there is an ordering of the clauses $(c_1, c_2, c_3, c_4)$ such that $c_4$ and $c_1$ share no literals; $c_2$ shares two literals with $c_1$ and one literal with $c_4$; and $c_3$ has the literals of $c_1$ and $c_4$ that do not appear in $c_2$.
    \item Form exactly two $K_{2,2}$ subgraphs and there is an ordering of the clauses $(c_1, c_2, c_3, c_4)$ such that $c_4$ shares one literal with $c_1$; $c_2$ has the common literal of $c_1$ and $c_4$, one more literal of $c_1$, and one more literal of $c_4$; and $c_3$ shares the common literal of $c_1$ and $c_4$, and has the literals of $c_1$ and $c_4$ that do not belong to $c_2$.
    \item Do not form any $K_{2,2}$ subgraphs in $G(\ClauseSet)$.
    \end{enumerate}
\end{lemma}

\tikzset{graph node/.style={circle, draw, minimum size=0.2cm}}
\begin{figure}[t]
    \centering
    \begin{minipage}[b]{.3\textwidth}
    \centering
    
        \begin{subfigure}[b]{0.3\textwidth}
         \begin{adjustwidth*}{}{-0.9em}
            \begin{tikzpicture}
            \node[graph node] (L1) at (0,0) {};
            \node[graph node] (L2) at (0,-1) {};
    
            \node[graph node] (R1) at (2,0) {};
            \node[graph node] (R2) at (2,-1) {};
    
            \draw (L1) -- node[midway, above, font=\footnotesize] {$e_1$} (R1);
            \draw (L1) -- node[near start, left,font=\footnotesize] {$e_2$} (R2);
            \draw (L2) -- node[near start, left,font=\footnotesize] {$e_3$} (R1);
            \draw (L2) -- node[midway, below,font=\footnotesize] {$e_4$} (R2);
            \end{tikzpicture}
        \end{adjustwidth*}
        \caption{}
        \label{fig:4-edges}
     \end{subfigure}
     
     \vfill
     \vspace{3mm}
     
     \begin{subfigure}[b]{0.3\textwidth}
     \begin{adjustwidth*}{}{-2.9em}
            \begin{tikzpicture}
                \node[graph node, label={left:\footnotesize $(a,b)$}] (L1) at (0,0) {};
                \node[graph node, label={left:\footnotesize $(d,f)$}] (L2) at (0,-1) {};
    
                \node[graph node, label={right:\footnotesize $g$}] (R1) at (2,0) {};
                \node[graph node,label={right:\footnotesize $h$}] (R2) at (2,-1) {};
    
                \draw (L1) --  (R1);
                \draw (L1) --  (R2);
                \draw (L2) --  (R1);
                \draw (L2) --  (R2);
            \end{tikzpicture}
        \end{adjustwidth*}
        \caption{}
        \label{fig:case-1}
     \end{subfigure}
     \vfill
     \vspace{3mm}
     
     \begin{subfigure}[b]{0.3\textwidth}
     \begin{adjustwidth*}{}{-2.9em}
            \begin{tikzpicture}
                \node[graph node, label={left:\footnotesize $(a,b)$}] (L1) at (0,0) {};
                \node[graph node, label={left:\footnotesize $(a,d)$}] (L2) at (0,-1) {};
    
                \node[graph node, label={right:\footnotesize $g$}] (R1) at (2,0) {};
                \node[graph node,label={right:\footnotesize $h$}] (R2) at (2,-1) {};
                
                \draw (L1) --  (R1);
                \draw (L1) --  (R2);
                \draw (L2) --  (R1);
                \draw (L2) --  (R2);
            \end{tikzpicture}
        \end{adjustwidth*}
         \caption{}
         \label{fig:case-2a}
     \end{subfigure}
     \vfill
     \vspace{3mm}
     
     \begin{subfigure}[b]{0.33\textwidth}
     \begin{adjustwidth*}{}{-2.9em}
            \begin{tikzpicture}
                \node[graph node, label={left:\footnotesize $(a,g)$}] (L1) at (0,0) {};
                \node[graph node, label={left:\footnotesize $(a,h)$}] (L2) at (0,-1) {};
                
                \node[graph node, label={right:\footnotesize $b$}] (R1) at (2,0) {};
                \node[graph node,label={right:\footnotesize $d$}] (R2) at (2,-1) {};
                
                \draw (L1) --  (R1);
                \draw (L1) --  (R2);
                \draw (L2) --  (R1);
                \draw (L2) --  (R2);
            \end{tikzpicture}
        \end{adjustwidth*}
         \caption{}
         \label{fig:case-2b}
     \end{subfigure}
     \end{minipage}
    \begin{subfigure}[b]{0.3\textwidth}
         \centering
            \begin{tikzpicture}
                \foreach \x [count=\xi] in {{(a,b)}, {(a,g)}, {(a,h)}, {(b,g)}, {(b,h)}, {(d,f)}, {(d,g)}, {(d,h)}, {(f,g)}, {(f,h)}}
                \node[graph node, label=left:\footnotesize $\x$] (L\xi) at (0,-\xi+1) {};
            
                \node[graph node, ] (L1) at (0,-1+1) {};
                \node[graph node, ] (L6) at (0,-6+1) {};
                
                \foreach \x [count=\xi] in {a, b, d, f, g, h}
                \node[graph node, label=right:\footnotesize $\x$] (R\xi) at (2,-\xi*1.5*1+0.75) {};
            
                \node[graph node, ] (R5) at (2,-6.75*1) {};
                \node[graph node, ] (R6) at (2,-8.25*1) {};
                
                \draw [] (L1) -- (R5);
                \draw [] (L1) -- (R6);
                \draw [] (L6) -- (R5);
                \draw [] (L6) -- (R6);
                \draw (L2) -- (R2);
                \draw (L3) -- (R2);
                \draw (L4) -- (R1);
                \draw (L5) -- (R1);
                \draw (L7) -- (R4);
                \draw (L8) -- (R4);
                \draw (L9) -- (R3);
                \draw (L10) -- (R3);
            \end{tikzpicture}
         \caption{}
         \label{fig:3-clauses-1-k22}
     \end{subfigure}
     \begin{subfigure}[b]{0.3\textwidth}
         \centering
            \begin{tikzpicture}
                \foreach \x [count=\xi] in {{(a,b)}, {(a,g)}, {(a,h)}, {(b,g)}, {(b,h)}, {(a,d)}, {(d,g)}, {(d,h)}}
                \node[graph node, label=left:\footnotesize $\x$] (L\xi) at (0,-\xi+1) {};
            
                \node[graph node, ] (L1) at (0,-1+1) {};
                \node[graph node, ] (L6) at (0,-6+1) {};
                \node[graph node, ] (L2) at (0,-2+1) {};
                \node[graph node, ] (L3) at (0,-3+1) {};
                
                \foreach \x [count=\xi] in {a, b, d, g, h}
                \node[graph node, label=right:\footnotesize $\x$] (R\xi) at (2,-\xi*1.5*1+0.97*1) {};
            
                \node[graph node, ] (R2) at (2,-2.03*1) {};
                \node[graph node, ] (R3) at (2,-3.53*1) {};
                \node[graph node,  ] (R4) at (2,-5.03*1) {};
                \node[graph node, ] (R5) at (2,-6.53*1) {};
                
                \draw (L1) -- (R4);
                \draw (L1) -- (R5);
                \draw (L6) -- (R4);
                \draw (L6) -- (R5);
                \draw (L2) -- (R2);
                \draw (L3) -- (R2);
                \draw (L2) -- (R3);
                \draw (L3) -- (R3);
                \draw (L4) -- (R1);
                \draw (L5) -- (R1);
                \draw (L7) -- (R1);
                \draw (L8) -- (R1);
            
            \end{tikzpicture}
         \caption{}
         \label{fig:3-clauses-2-k22}
     \end{subfigure}
    \caption{Possible structures formed by four clauses in the corresponding bipartite multi-graph. In subfigure (e), there is one $K_{2,2}$ subgraph formed by nodes $(a,b)$, $(d,g)$, $g$, and $h$ and their incident edges. In subfigure (f), there are two $K_{2,2}$ subgraphs; one formed by the nodes $(a,b)$, $(a,d)$, $g$, and $h$ and their incident edges and one formed by the nodes $(a,g)$, $(a,h)$, $b$, and $d$ and their incident edges.}
    \label{fig:k22}
\end{figure}
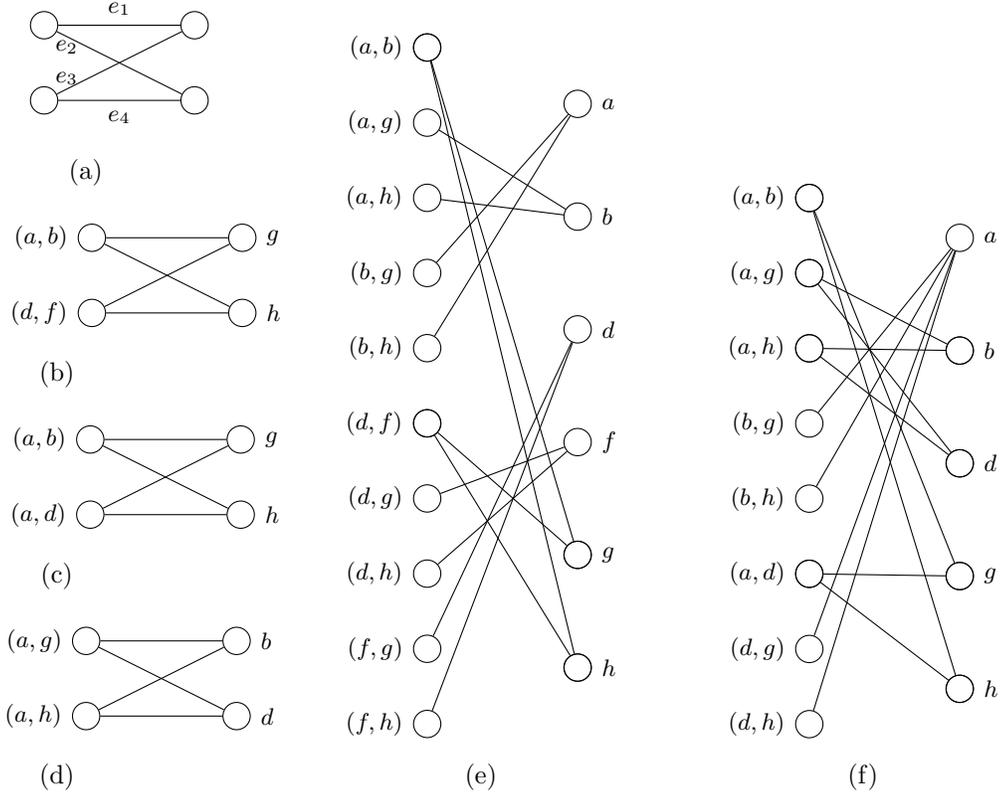

\begin{proof}
Consider a $K_{2,2}$ subgraph $H$ of $G(\ClauseSet)$ with the four edges $e_1$, $e_2$, $e_3$, and $e_4$ as in Figure~\ref{fig:4-edges} and assume that they correspond to clauses $c_1$, $c_2$, $c_3$, and $c_4$, respectively. Then, the two right nodes in $H$ correspond to different literals, e.g., $g$ and $h$. The two left nodes in $H$ may correspond to either two disjoint pairs of literals, e.g. $\{a, b\}$ and $\{d, f\}$ (see Figure~\ref{fig:case-1}) or to two pairs sharing a literal, e.g., $\{a, b\}$ and $\{a, d\}$ (see Figure~\ref{fig:case-2a}). So, using the letters $a$, $b$, $d$, $f$, $g$, $h$ to denote distinct literals, we distinguish between the two cases:

\begin{itemize}
    \item{\bf Case 1.} $c_1=(a,b,g)$, $c_2=(a,b,h)$, $c_3=(d,f,g)$, and $c_4=(d,f,h)$. Then, the twelve edges corresponding to $c_1$, $c_2$, $c_3$, and $c_4$ in $G(\ClauseSet)$ are depicted in Figure~\ref{fig:3-clauses-1-k22}. We observe that the $K_{2,2}$ of Figure~\ref{fig:case-1} is the only $K_{2,2}$ subgraph defined by $c_1$, $c_2$, $c_3$, and $c_4$.

    \item{\bf Case 2.} $c_1=(a,b,g)$, $c_2=(a,b,h)$, $c_3=(a,d,g)$, and $c_4=(a,d,h)$. Now, the twelve edges corresponding to $c_1$, $c_2$, $c_3$, and $c_4$ in $G(\ClauseSet)$ are depicted in Figure~\ref{fig:3-clauses-2-k22}. Observe that they form two $K_{2,2}$ subgraphs in $G(\ClauseSet)$ (the two depicted in Figures~\ref{fig:case-2a} and~\ref{fig:case-2b}).
\end{itemize}
These are the only cases in which four clauses can form $K_{2,2}$ subgraphs in graph $G(\ClauseSet)$.
\end{proof}

\begin{lemma}\label{lem:4-wise-upper-bound-argument}
For any $\Distribution\in\DistFamily_4(n,m)$, it holds that $\Ex[C\sim\Distribution]{\kappa(C)}\leq \frac{81m^4}{64n^6}+\frac{729m^3}{32n^4}$.
\end{lemma}

\begin{proof}
Let $p_1$ and $p_2$ denote the probabilities that a set of specific four clauses of $\ClauseSet\sim\Distribution$ respectively falls into case 1 and 2 from Lemma~\ref{lem:structure} (as $\Distribution\in\DistFamily_4(n,m)$ the probabilities $p_1$ and $p_2$ are the same for any four clauses). By linearity of expectation and 
Lemma~\ref{lem:structure} we have
    \begin{align}\label{eq:upper-bound-for-kappa}
        \Exlong[C\sim \Distribution]{\kappa(C)}&= \binom{m}{4} \cdot (p_1+2\cdot p_2).
    \end{align}
To bound $p_1$, consider an arbitrary ordering $(c_1, c_2, c_3, c_4)$ of the four clauses  (there are $4!=24$ orderings). Given $c_1$ and $c_4$ have different literals (happens with probability close to $1$), there are $9$ ways to select two literals from $c_1$ and one literal from $c_4$ in $c_2$ and then only one way to select the remaining literals in $c_3$. 
Note that there are exactly $4$ orderings $(c_1,c_2,c_3,c_4)$, $(c_2,c_1,c_4,c_3)$, $(c_3,c_4,c_1,c_2)$, and $(c_4,c_1,c_4,c_3)$ of $\{c_1,c_2,c_3,c_4\}$ that meet the conditions of Case $1$ from Lemma~\ref{lem:structure}. 
As we can assume without loss of generality that these four clauses are selected independently (due to $4$-clause independence of $\Distribution$), we have
\begin{align}\label{eq:bound-for-p_1}
    p_1 &\leq 4!\cdot \frac{9}{8\binom{n}{3}}\cdot \frac{1}{8\binom{n}{3}}\cdot \frac{1}{4}=\frac{243}{8n^2(n-1)^2(n-2)^2} \leq \frac{243}{8n^6}+\frac{1215}{4n^7}.
\end{align}
The last inequality follows after verifying that $\frac{1}{(n-1)^2(n-2)^2}\leq \frac{1}{n^4}+\frac{10}{n^5}$ when $n\geq 10$.

We now bound probability $p_2$ that the four clauses meet the conditions of Case 2 from Lemma~\ref{lem:structure}. Let us fix an  ordering $(c_1, c_2, c_3, c_4)$ of the four clauses ($4!=24$ orderings in total). For a given clause $c_1$ there are $3$ choices of the shared literal between $c_1$ and $c_4$, i.e., the probability that $c_4$ meets the conditions of Case 2 is at most $\frac{3\cdot 4\binom{n-1}{2}}{8\binom{n}{3}}$. For fixed $c_1$ and $c_4$, there are only four choices for $c_2$; and for given $c_1$, $c_2$, and $c_4$ there is a unique choice for $c_3$. There are exactly $4$ orderings $(c_1,c_2,c_3,c_4)$, $(c_2,c_1,c_4,c_3)$, $(c_3,c_4,c_1,c_2)$, and $(c_4,c_1,c_4,c_3)$ of $\{c_1,c_2,c_3,c_4\}$ that meet conditions of Case 2. As we can assume without loss of generality that these four clauses are selected independently (due to $4$-clause independence of $\Distribution$), we have
\begin{align}\label{eq:bound-for-p_2}
    p_2 &\le \frac{4!}{4}\cdot \frac{3\cdot 4\binom{n-1}{2}}{8\binom{n}{3}} \cdot \frac{4}{8\binom{n}{3}}\cdot \frac{1}{8\binom{n}{3}} = \frac{243}{4n^3(n-1)^2(n-2)^2}\leq \frac{243}{2n^7},
\end{align}
where the last inequality holds, as $(n-1)^2(n-2)^2\geq \frac{n^4}{2}$ for $n\ge 10$. Finally, by substituting
\eqref{eq:bound-for-p_1} and \eqref{eq:bound-for-p_2} into \eqref{eq:upper-bound-for-kappa} we get
\begin{align*}
    \Exlong[C\sim \Distribution]{\kappa(C)} &\leq \frac{81m^4}{64n^6}+\frac{729m^4}{32n^7} \leq \frac{81m^4}{64n^6}+\frac{729m^3}{32n^4},
\end{align*}
as desired (the last inequality follows, since $m\le \sqrt{10}\cdot n^{5/2}\leq n^3$).
\end{proof}

\paragraph{Lower-bounding $\Ex[\ClauseSet\sim \Distribution]{\kappa(\ClauseSet)}$.} For a $3$-SAT instance $C$ with $m$ clauses over $n$ variables, let $\tilC$ be the maximal subset of clauses in $\ClauseSet$ with removed duplicates (identical clauses). The (simple) graph $G(\tilC)$ is the underlying simple subgraph of the multi-graph $G(\ClauseSet)$. We denote by $\tilm$ the number of clauses in $\tilC$. We next obtain two lower bounds on $\kappa(\tilC)$ in terms of $n$ and $\tilm$: (i) for arbitrary $\tilC$ in Lemmas~\ref{lem:lower-bound-4-wise-any} and better bound (ii) for the case when $\tilC$ is SAT. These bounds only exploit the structure of the graph $G(\tilC)$ and thus
do not depend on any properties of $\Distribution$. 
We then show in Lemma~\ref{lem:tilm-vs-m} how to relate expectation of $\tilm$ ($\tilm\le m$) to $m$.
Finally, we combine all these bounds in Lemma~\ref{lem:lower-bound-for-kappa} to relate $\Ex[\ClauseSet\sim \Distribution]{\kappa(\ClauseSet)}$ with $\Prx[\ClauseSet\sim \Distribution]{SAT(\ClauseSet)}$.

Let us now define some notation used in the proof of both Lemmas~\ref{lem:lower-bound-4-wise-any} and \ref{lem:lower-bound-4-wise-sat}. We denote by $R_2$ the set of all pairs of distinct nodes in the right node set $R$ of graph $G(C)=(R\cup L,E)$. Thus, $|R_2|=\binom{2n}{2}>|L|$. Then, for a pair $(v,v')\in R_2$, let $B_{v,v'}$ denote the set of nodes in $L$ that are adjacent to both nodes $v$ and $v'$ in $R$.

In the proofs of Lemmas~\ref{lem:lower-bound-4-wise-any} and \ref{lem:lower-bound-4-wise-sat}, we extensively use the following observation.

\begin{claim}\label{claim:multiple-concave-functions}
    Let $x_1, ..., x_k\geq 0$. Then $\sum_{i=1}^k{x_i^2}\geq \frac{1}{k}\left(\sum_{i=1}^k{x_i}\right)^2$.
\end{claim}

We are ready to prove our first lower bound on the number $\kappa(\tilC)$.
\begin{lemma}\label{lem:lower-bound-4-wise-any}
    Let $\tilC$ be an instance over $n$ variables which consists of $\tilm$ distinct clauses. Then, 
    $$\kappa(\tilC)\geq \frac{81\tilm^4}{64n^6}-\frac{27\tilm^3}{8n^4}.$$
\end{lemma}

\begin{proof}
Observe that the lemma trivially holds if $\tilm\leq 8n^2/3$, as the RHS of the inequality in its statement is non-positive. In the following, we consider the case $\tilm>8n^2/3$. Notice that this implies that $\tilm>4|R_2|/3$ and $\tilm>4|L|/3$; these facts will be useful later in the proof.

Let $v$ and $v'$ be two (different) nodes in $R$ side of $G(\tilC)$. The number of different $K_{2,2}$ subgraphs in $G(\tilC)$ with vertices $v$ and $v'$ is $\binom{|B_{v,v}|}{2}$. Thus, using Claim~\ref{claim:multiple-concave-functions}, $\kappa(\tilC)$ can be bounded as
\begin{align}\nonumber
\kappa(\tilC) &=\sum_{(v,v')\in R_2}{\binom{|B_{v,v'}|}{2}} = \frac{1}{2}\sum_{(v,v')\in R_2}{|B_{v,v'}|^2}-\frac{1}{2}\sum_{(v,v')\in R_2}{|B_{v,v'}|}\\\label{eq:lower-bound-for-kappa-1}
&\geq \frac{1}{2|R_2|}\left(\sum_{(v,v')\in R_2}{|B_{v,v'}|}\right)^2-\frac{1}{2}\sum_{(v,v')\in R_2}{|B_{v,v'}|}.
\end{align}
Now, notice that the quantity $\sum_{(v,v')\in R_2}{|B_{v,v'}|}$ is the total number of distinct pairs of different nodes in $R$ that the nodes in $L$ have. Thus, denoting the degree of node $u\in L$ by $d_u$ (recall that $\sum_{u\in L}{d_u}=3\tilm$), we get
\begin{align}\nonumber
    \sum_{(v,v')\in R_2}{|B_{v,v'}|} &=\sum_{u\in L}{\binom{d_u}{2}}
    =\frac{1}{2}\sum_{u\in L}{d_u^2}-\frac{1}{2}\sum_{u\in L}{d_u}\\\label{eq:lower-bound-for-kappa-2}
    &\geq \frac{1}{2|L|}\left(\sum_{u\in L}{d_u}\right)^2 - \frac{1}{2}\sum_{u\in L}{d_u}
    =\frac{9\tilm^2}{2|L|}-\frac{3\tilm}{2}.
\end{align}
The inequality follows by Claim~\ref{claim:multiple-concave-functions}. We now aim to substitute $\sum_{(v,v')\in R_2}{|B_{v,v'}|}$ with $\frac{9\tilm^2}{2|L|}-\frac{3\tilm}{2}$ at the RHS of Equation (\ref{eq:lower-bound-for-kappa-1}) and get the desired bound on $\kappa(\tilC)$. However, before doing so, we need to make sure that the RHS of Equation (\ref{eq:lower-bound-for-kappa-1}) is increasing for the range of values Equation (\ref{eq:lower-bound-for-kappa-2}) defines for $\sum_{(v,v')\in R_2}{|B_{v,v'}|}$. Indeed, we have
\begin{align}\label{eq:lower-bound-for-kappa-observation}   
\frac{9\tilm^2}{2|L|}-\frac{3\tilm}{2} &= \frac{3\tilm}{2}\left(\frac{3\tilm}{|L|}-1\right) > 6|R_2|\geq |R_2|/2.
\end{align}
The inequality follows by the facts $\tilm> 4|R_2|/3$ (and, hence, $\frac{3\tilm}{2}\geq 2|R_2|$) and $\tilm>4|L|/3$ (and, hence, $\frac{3\tilm}{|L|}-1>3$). Now, observe that the function $\frac{1}{2|R_2|}x^2-\frac{1}{2}x$ is increasing for $x\geq |R_2|/2$. Thus, due to Equation (\ref{eq:lower-bound-for-kappa-2}) and by our observation in Equation (\ref{eq:lower-bound-for-kappa-observation}), Equation (\ref{eq:lower-bound-for-kappa-1}) yields 
\begin{align*}
\kappa(\tilC) &\geq \frac{1}{2|R_2|}\left(\frac{9\tilm^2}{2|L|}-\frac{3\tilm}{2}\right)^2-\frac{1}{2}\left(\frac{9\tilm^2}{2|L|}-\frac{3\tilm}{2}\right)=\frac{81\tilm^4}{8|R_2||L|^2}\left(1-\frac{|L|}{3\tilm}\right)^2-\frac{9\tilm^2}{4|L|}+\frac{3\tilm}{4}\\
&\geq \frac{81\tilm^4}{8|R_2||L|^2}-\frac{27\tilm^3}{4|R_2||L|}-\frac{9\tilm^2}{4|L|}
\geq \frac{81\tilm^4}{8|R_2||L|^2}-\frac{27\tilm^3}{2|R_2||L|}= \frac{27\tilm^3}{2|R_2||L|}\left(\frac{3\tilm}{4|L|}-1\right)\\
&\geq \frac{27\tilm^3}{8n^4}\left(\frac{3\tilm}{8n^2}-1\right)= \frac{81\tilm^4}{64n^6}-\frac{27\tilm^3}{8n^4},
\end{align*}
as desired. The second inequality follows by the property $(1-x)^2\geq 1-2x$. The third one follows since $\tilm > 4|R_2|/3 > |R_2|/3$ (and, hence, $\frac{9\tilm^2}{4|L|}\leq \frac{27\tilm^3}{4|R_2||L|}$). The last inequality follows by the facts $|R_2|=\binom{2n}{2}<2n^2$ and $|L|=4\binom{n}{2}<2n^2$.
\end{proof}

We now present a second lower bound on the number $\kappa(\tilC)$, specifically for the case of a satisfiable instance $C$.

\begin{lemma}\label{lem:lower-bound-4-wise-sat}
        Let $C$ be a statisfiable instance with $n$ variables and $\tilm$ distinct clauses. Then
    $$\kappa(\tilC) \geq  \frac{82\tilm^4}{64n^6}-\frac{123\tilm^3}{16n^4}.$$
\end{lemma}

\begin{proof}
Observe that the lemma trivially holds if $\tilm\leq 6n^2$, as the RHS of the inequality in its statement is non-positive. In the following, we consider the case $\tilm>6n^2$.

We begin with some additional definitions. Consider an assignment $\sigma$ satisfying $C$ (and, thus, satisfying $\tilC$). Let $L_S$ be the subset of $L$ consisting of nodes identified by two literals, at least one of which is true according to $\sigma$. Let $L_U=L\setminus L_S$. Similarly, let $R_S$ be the subset of $R$ consisting of nodes identified by a literal which is true under $\sigma$ and define $R_U=R\setminus R_S$. Clearly, $|L_S|=3|L|/4$, $|L_U|=|L|/4$, and $|R_S|=|R_U|=|R|/2$. Finally, let $R_2^{SU}$ be the subset of $R_2$ consisting of pairs in $R_S\times R_U$, i.e., having a node from $R_S$ and a node from $R_U$. Clearly, $|R_2^{SU}|=n^2$ and $|R_2\setminus R_2^{SU}|=n^2-n$.

Define $P=\sum_{(v,v')\in R_2^{SU}}{|B_{v,v'}|}$ and $Q=\sum_{(v,v')\in R_2\setminus R_2^{SU}}{|B_{v,v'}|}$. We will bound the number $\kappa(\tilC)$ using 
\begin{align}\nonumber
\kappa(\tilC) &=\sum_{(v,v')\in R_2}{\binom{|B_{v,v'}|}{2}}\\\nonumber 
&= \frac{1}{2}\sum_{(v,v')\in R_2^{SU}}{|B_{v,v'}|^2}-\frac{1}{2}\sum_{(v,v')\in R_2^{SU}}{|B_{v,v'}|}+\frac{1}{2}\sum_{(v,v')\in R_2\setminus R_2^{SU}}{|B_{v,v'}|^2}-\frac{1}{2}\sum_{(v,v')\in R_2\setminus R_2^{SU}}{|B_{v,v'}|}\\\nonumber
&\geq \frac{1}{2|R_2^{SU}|}\left(\sum_{(v,v')\in R_2^{SU}}{|B_{v,v'}|}\right)^2-\frac{1}{2}\sum_{(v,v')\in R_2^{SU}}{|B_{v,v'}|}\\\nonumber
&\quad +\frac{1}{2|R\setminus R_2^{SU}|}\left(\sum_{(v,v')\in R_2\setminus R_2^{SU}}{|B_{v,v'}|}\right)^2-\frac{1}{2}\sum_{(v,v')\in R_2\setminus R_2^{SU}}{|B_{v,v'}|}\\\nonumber
&\geq \frac{1}{2n^2}(P^2+Q^2)-\frac{1}{2}(P+Q)\\\label{eq:lower-bound-for-kappa-sat-1}
&= \frac{1}{4n^2}(P+Q)(P+Q-2n^2)+\frac{1}{4n^2}(Q-P)^2.
\end{align}
The first inequality follows by Claim~\ref{claim:multiple-concave-functions} and the second one by the fact $|R_2\setminus R_2^{SU}|\leq |R_2^{SU}|=n^2$.

For a node $u\in L_S$, denote by $\alpha_u$ and $\beta_u$ the number of nodes in $R_S$ and $R_U$ that are adjacent to node $u$ in $G(\tilC)$. For a node $u\in L_U$, denote by $\gamma_u$ the number of nodes in $R_S$ that are adjacent to node $u$ in $G(\tilC)$; notice that $u$ cannot be adjacent to any node of $R_U$ as such an edge would correspond to an unsatisfied clause under $\sigma$. These definitions imply that
\begin{align}\label{eq:lower-bound-for-kappa-sat-3m}
    \sum_{u\in L_S}{(\alpha_u+\beta_u)}+\sum_{u\in L_U}{\gamma_u} &= 3\tilm.
\end{align}

Observe that for each node $u\in L_S$, there are $\alpha_u\beta_u$  node pairs in $R_2^{SU}$ and $\binom{\alpha_u}{2}+\binom{\beta_u}{2}$ node pairs in $R_2 \setminus R_2^{SU}$ to which node $u$ is adjacent. Also, for each node $u\in L_U$, there are $\binom{\gamma_u}{2}$ node pairs in $R_2\setminus R_2^{SU}$ to which node $u$ is adjacent. Thus, using the definition of $B_{v,v'}$, we get
\begin{align}\label{eq:lower-bound-for-kappa-sat-P}
    P&=\sum_{(v,v')\in R_2^{SU}}{|B_{v,v'}|} = \sum_{u\in L_S}{\alpha_u\beta_u}\\\label{eq:lower-bound-for-kappa-sat-Q}
    Q&=\sum_{(v,v')\in R_2\setminus R_2^{SU}}{|B_{v,v'}|} = \sum_{u\in L_S}{\left(\binom{\alpha_u}{2}+\binom{\beta_u}{2}\right)}+\sum_{u\in L_U}{\binom{\gamma_u}{2}}.
\end{align}
To complete the proof, we will bound the RHS of Equation (\ref{eq:lower-bound-for-kappa-sat-1}) with $P$ and $Q$ defined as in Equations (\ref{eq:lower-bound-for-kappa-sat-P}) and (\ref{eq:lower-bound-for-kappa-sat-Q}), under the constraint imposed by Equation (\ref{eq:lower-bound-for-kappa-sat-3m}). 

Observe that 
\begin{align}
    \nonumber
    P+Q&=\sum_{u\in L_S}\InParentheses{\binom{\alpha_u}{2}+\binom{\beta_u}{2}+\alpha_u\beta_u}+\sum_{u\in L_U}\binom{\gamma_u}{2}\\
    \nonumber
    &=\frac{1}{2}\sum_{u\in L_S}(\alpha_u+\beta_u)^2+\frac{1}{2}\sum_{u\in L_U}\gamma_u^2-\frac{1}{2}\sum_{u\in L_S}(\alpha_u+\beta_u)-\frac{1}{2}\sum_{u\in L_U}\gamma_u\\
    \nonumber
    &\geq \frac{1}{2|L_S|}\left(\sum_{u\in L_S}(\alpha_u+\beta_u)\right)^2+\frac{1}{2|L_U|}\left(\sum_{u\in L_U}\gamma_u\right)^2-\frac{1}{2}\sum_{u\in L_S}(\alpha_u+\beta_u)-\frac{1}{2}\sum_{u\in L_U}\gamma_u\\
    \nonumber
    &=\frac{1}{2|L_S|}\left(3\tilm-\sum_{u\in L_U}\gamma_u\right)^2+\frac{1}{2|L_U|}\left(\sum_{u\in L_U}\gamma_u\right)^2-\frac{3\tilm}{2}\\
    \nonumber
    &\geq\frac{1}{3n^2}\left(3\tilm-\sum_{u\in L_U}\gamma_u\right)^2+\frac{1}{n^2}\left(\sum_{u\in L_U}\gamma_u\right)^2-\frac{3\tilm}{2}\\
    \label{eq:lower-bound-for-kappa-sat-P-plus-Q}
    &=\frac{4}{3n^2}\left(\sum_{u\in L_U}\gamma_u\right)^2-\frac{2\tilm}{n^2}\sum_{u\in L_U}\gamma_u+\frac{3\tilm^2}{n^2}-\frac{3\tilm}{2}.
\end{align}
The first inequality is due to Claim \ref{claim:multiple-concave-functions}, the third equation is due to Equation \eqref{eq:lower-bound-for-kappa-sat-3m}, and the second inequality is due to the facts $|L_S|=\frac{3}{4}|L|$, $|L_U|=\frac{1}{4}|L|$, and $|L|=4\binom{n}{2}<2n^2$. I.e., for the quadratic function $f(x)\eqdef\frac{4}{3n^2}x^2-\frac{2\tilm}{n^2}x+\frac{3\tilm^2}{n^2}-\frac{3\tilm}{2}$ the equation \eqref{eq:lower-bound-for-kappa-sat-P-plus-Q} reads as $P+Q\ge f(x_0)$ for $x_0=\sum_{u \in L_U} \gamma_u$. This function will also be useful later. 
Similarly, we have
\begin{align}
    \nonumber
    Q-P&=\sum_{u\in L_S}\InParentheses{\binom{\alpha_u}{2}+\binom{\beta_u}{2}-\alpha_u\beta_u}+\sum_{u\in L_U}\binom{\gamma_u}{2}\\
    \nonumber
    &=\frac{1}{2}\sum_{u\in L_S}(\alpha_u-\beta_u)^2+\frac{1}{2}\sum_{u\in L_U}\gamma_u^2-\frac{1}{2}\sum_{u\in L_S}(\alpha_u+\beta_u)-\frac{1}{2}\sum_{u\in L_U}\gamma_u\\
    \nonumber
    &\geq \frac{1}{2}\sum_{u\in L_U}\gamma_u^2-\frac{1}{2}\sum_{u\in L_S}(\alpha_u+\beta_u)-\frac{1}{2}\sum_{u\in L_U}\gamma_u\\
    \nonumber
    &\geq \frac{1}{2|L_U|}\left(\sum_{u\in L_U}\gamma_u\right)^2-\frac{1}{2}\sum_{u\in L_U}(\alpha_u+\beta_u)-\frac{1}{2}\sum_{u\in L_U}\gamma_u\\
    \label{eq:lower-bound-for-kappa-sat-Q-minus-P}
    &=\frac{1}{2|L_U|}\left(\sum_{u\in L_U}\gamma_u\right)^2-\frac{3\tilm}{2}\geq \frac{1}{n^2}\left(\sum_{u\in L_U}\gamma_u\right)^2-\frac{3\tilm}{2}.
\end{align}

We conclude the proof by considering two cases for $x_0=\sum_{u \in L_U} \gamma_u$.

\paragraph{Case 1: $x_0=\sum_{u \in L_U}\gamma_u \geq \frac{\tilm}{2}$.} Notice that $f(x)$ is minimized for $x=\frac{3\tilm}{4}$ at the value $\frac{9\tilm^2}{4n^2}-\frac{3\tilm}{2}$, which is greater than $n^2$ due to the fact $\tilm>6n^2$. This observation together with Equation \eqref{eq:lower-bound-for-kappa-sat-P-plus-Q} yields $P+Q\ge f(x_0)>n^2$. Now notice that the quadratic function $x(x-2n^2)$ is increasing for $x>n^2$, and using the bound $P+Q\ge f(x_0)\ge \frac{9\tilm^2}{4n^2}-\frac{3\tilm}{2}$, we get

\begin{align}
    \label{eq:lower-bound-for-kappa-sat-P-plus-Q-case-1}
    (P+Q)(P+Q-2n^2)\geq \InParentheses{\frac{9\tilm^2}{4n^2}-\frac{3\tilm}{2}}\InParentheses{\frac{9\tilm^2}{4n^2}-\frac{3\tilm}{2}-2n^2}\geq \frac{81\tilm^4}{16n^4}-\frac{27\tilm^3}{4n^2}-\frac{9\tilm^2}{4}.
\end{align}
As $x_0\ge \tilm/2$, \eqref{eq:lower-bound-for-kappa-sat-Q-minus-P} yields $Q-P\geq \frac{\tilm^2}{4n^2}-\frac{3\tilm}{2}$, which is positive because $\tilm > 6n^2$. Thus,
\begin{align}
    \label{eq:lower-bound-for-kappa-sat-Q-minus-P-case-1}
    (Q-P)^2 \geq \left(\frac{\tilm^2}{4n^2}-\frac{3\tilm}{2}\right)^2=\frac{\tilm^4}{16n^4}-\frac{3\tilm^3}{4\tilm^2}+\frac{9\tilm^2}{4}.
\end{align}
By equations \eqref{eq:lower-bound-for-kappa-sat-1}, \eqref{eq:lower-bound-for-kappa-sat-P-plus-Q-case-1}, and \eqref{eq:lower-bound-for-kappa-sat-Q-minus-P-case-1}, we get
\begin{align*}
    \kappa(\tilC)\geq \frac{82\tilm^4}{64n^6}-\frac{15\tilm^3}{8n^4}\geq\frac{82\tilm^4}{64n^6}-\frac{123\tilm^3}{16n^4}
\end{align*}
as desired.

\paragraph{Case 2: $x_0=\sum_{u \in L_U} \gamma_u \leq \frac{\tilm}{2}$.} Notice that the quadratic function $f(x)$ is decreasing in $[0,\frac{3\tilm}{4})$. Thus, $f(x_0)\ge\frac{7\tilm^2}{3n^2}-\frac{3\tilm}{2}$, which is higher than $n^2$, as $\tilm>6n^2$. Hence, $P+Q\ge f(x_0)>n^2$. As the quadratic function $x(x-2n^2)$ is increasing for $x>n^2$ and $P+Q \geq \frac{7\tilm^2}{3n^2}-\frac{3\tilm}{2}>n^2$, we get

\begin{align}
    \nonumber
    (P+Q)(P+Q-2n^2)&\geq \left(\frac{7\tilm^2}{3n^2}-\frac{3\tilm}{2}\right)\left(\frac{7\tilm^2}{3n^2}-\frac{3\tilm}{2}-2n^2\right)\\
    \label{eq:lower-bound-for-kappa-sat-P-plus-Q-case-2}
    &\geq\frac{49\tilm^4}{9n^4}-\frac{7\tilm^3}{n^2}-\frac{29\tilm^2}{12}\geq\frac{82\tilm^4}{16n^4}-\frac{123\tilm^3}{4n^2}.
\end{align} 

The last inequality holds as $\tilm>6n^2$. The lemma now follows by Equations \eqref{eq:lower-bound-for-kappa-sat-1} and \eqref{eq:lower-bound-for-kappa-sat-P-plus-Q-case-2}.
\end{proof}

The next lemma relates $\tilm$ with $m$ and $n$.
\begin{lemma}\label{lem:tilm-vs-m}
    $\Ex[C\sim \Distribution]{\tilm^4} \geq m^4-\frac{125}{6}\cdot m^3\cdot n^2$.
\end{lemma}

\begin{proof}
    For two clauses $c$ and $c'$ of $C$, let $X(c,c')$ be the binary random variable indicating whether $c=c'$ (then $X(c,c')=1$) or not (then, $X(c,c')=0$).     
    We first observe that 
    \begin{equation}
    \label{eq:m_tilm_count}
        m \le \tilm + \sum_{\{c,c'\}\in\binom{\ClauseSet}{2}}X(c,c').
    \end{equation} Indeed, let  $k=k(\theta)$ be the multiplicity of each clause type $\theta$ in $\ClauseSet$, then the left hand side $m=\sum_{\theta}k(\theta)$. For any value of $k$, the type $\theta$ contributes at least $k(\theta)$ to the right hand side of \eqref{eq:m_tilm_count}: (i) for $k=1$ we count one $\theta$ in $\tilm$; (ii) for $k=2$ we count $\theta$ once in $\tilm$ and another time in $X(c,c')$ for $c=\theta=c'$; (iii) for $k\ge 3$ the terms $X(c,c')$ with $c=\theta$ and $c'=\theta$ contribute $\binom{k}{2}\ge k$ to the right hand side of \eqref{eq:m_tilm_count}. After taking expectation over $\ClauseSet\sim\Distribution$, we get
    \begin{align*}
    \Exlong[\ClauseSet\sim\Distribution]{\tilm} &\ge 
    m-\sum_{\{c,c'\}}\Prlong[\ClauseSet\sim\Distribution]{c=c'}=
    m-\binom{m}{2}\cdot \frac{1}{8\binom{n}{3}}\ge m-\frac{3m^2}{8n(n-1)(n-2)}\geq m-\frac{25m^2}{48n^3}\\
        &=m\left(1-\frac{25m}{48n^3}\right).
    \end{align*}

    Thus, using Jensen's inequality, the fact $(1-x)^4\geq 1-4x$, and the assumption $m\leq \sqrt{10}\cdot n^{5/2}$, we have
    \begin{align*}
        \Exlong[\ClauseSet\sim\Distribution]{\tilm^4}&\geq \Exlong[\ClauseSet\sim\Distribution]{\tilm}^4 \geq m^4\left(1-\frac{25m}{48n^3}\right)^4 \geq m^4-\frac{25m^5}{48n^3} \geq m^4-\frac{125}{6}\cdot m^3n^2,
    \end{align*}
    as desired. 
\end{proof}
We prove a lower bound on $\Ex{\kappa(\ClauseSet)}$ using the last three lemmas.
\begin{lemma}\label{lem:lower-bound-for-kappa}
    For any $\Distribution\in\DistFamily_4(n,m)$, we have $\Ex[C\sim \Distribution]{\kappa(C)}\geq \frac{81m^4}{64n^6}-\frac{1215m^3}{32n^4}+\frac{m^4}{64n^6}\cdot \Prx[\ClauseSet\sim\dist]{SAT(\ClauseSet)}$.    
\end{lemma}

\begin{proof}
We prove the lemma with the following derivation:
    \begin{align*}
        \Exlong[\ClauseSet\sim\Distribution]{\kappa(\ClauseSet)}&\geq \Exlong[\ClauseSet\sim\Distribution]{\kappa(\tilC)} \\
        &=\Exlong[\ClauseSet\sim\Distribution]{\kappa(\tilC)|SAT(\ClauseSet)}\cdot \Prlong[\ClauseSet\sim\Distribution]{SAT(\ClauseSet)}\\
        &\quad +\Exlong[\ClauseSet\sim\Distribution]{\kappa(\tilC)|\neg SAT(\ClauseSet))}\cdot \Prlong[\ClauseSet\sim\Distribution]{\neg SAT(\ClauseSet)}\\
        &\geq \Exlong[\ClauseSet\sim\Distribution]{\frac{82\tilm^4}{64n^6}-\frac{123\tilm^3}{16n^4}\bigg|SAT(\ClauseSet)}\cdot \Prlong[\ClauseSet\sim\Distribution]{SAT(\ClauseSet)}\\
        &\quad + \Exlong[\ClauseSet\sim\Distribution]{\frac{81\tilm^4}{64n^6}-\frac{27\tilm^3}{8n^4}\bigg|\neg SAT(\ClauseSet)}\cdot \Prlong[\ClauseSet\sim\Distribution]{\neg SAT(\ClauseSet)}\\
        &\geq \Exlong[\ClauseSet\sim\Distribution]{\frac{82\tilm^4}{64n^6}-\frac{123\tilm^3}{16n^4}}-\Exlong[\ClauseSet\sim\Distribution]{\frac{\tilm^4}{64n^6}\bigg|\neg SAT(\ClauseSet)}\cdot \Prlong[\ClauseSet\sim\Distribution]{\neg SAT(\ClauseSet)}\\
        &\geq \frac{82}{64n^6}\Exlong[\ClauseSet\sim\Distribution]{\tilm^4}-\frac{123m^3}{16n^4}-\frac{m^4}{64n^6}\cdot\Prlong[\ClauseSet\sim\Distribution]{\neg SAT(\ClauseSet)}\\
        &\geq \frac{82m^4}{64n^6}-\frac{861m^3}{32n^4}-\frac{123m^3}{16n^4}-\frac{m^4}{64n^6}\cdot\Prlong[\ClauseSet\sim\Distribution]{\neg SAT(\ClauseSet)}\\
        &= \frac{81m^4}{64n^6}-\frac{1215m^3}{32n^4}+\frac{m^4}{64n^6}\cdot\Prlong[\ClauseSet\sim\Distribution]{SAT(\ClauseSet)},
    \end{align*}
    as desired. The second inequality follows by Lemmas~\ref{lem:lower-bound-4-wise-sat} and~\ref{lem:lower-bound-4-wise-any}, the fourth one by the fact $\tilm\leq m$, and the fifth one by Lemma~\ref{lem:tilm-vs-m}.
\end{proof}

\paragraph{Putting everything together.} Lemmas~\ref{lem:4-wise-upper-bound-argument} and~\ref{lem:lower-bound-for-kappa} yield
\begin{align*}
    \frac{81m^4}{64n^6}-\frac{1215m^3}{32n^4}+\frac{m^4}{64n^6}\cdot \Prlong[\ClauseSet\sim\Distribution]{SAT(\ClauseSet)} &\leq \Exlong[\ClauseSet\sim\Distribution]{\kappa(C)} \leq \frac{81m^4}{64n^6}+\frac{729m^3}{32n^4},
\end{align*}
which implies the desired bound $\Prx[\ClauseSet\sim\dist]{SAT(\ClauseSet)}\leq \frac{4288\cdot n^2}{m}$, completing the proof of Theorem~\ref{thm:upper-bound-result-4wise}.
\end{proof}

\section{Bounds on the lower satisfiability threshold}
    \label{sec:SmallUnsat-proof}
    In this section, we present our upper bound on the lower satisfiability threshold and explain why we believe that it is the tight bound for every degree of independence. 
\begin{theorem}
For every integer $k\geq 2$, $\LST_k(n)=\Omega(n^{1-1/k})$.
    \label{thm:lower-bound-k-wise-unsat}
\end{theorem}

\begin{proof}
We prove the theorem by showing that for every integers $n\geq 3$ and $m\leq \frac{1}{12}\cdot n^{1-1/k}$, any $k$-clause independent distribution $\Distribution\in \DistFamily_k(n,m)$ satisfies $\Prx[C\sim \Distribution]{\neg SAT(C)}\leq 1/3$. 

Let $G(\ClauseSet)=(V,E)$ be a $3$-uniform hypergraph defined by an instance $\ClauseSet$ drawn from a $k$-clause independent distribution $\Distribution$, $V=X(n)$ and $E=\set{\GetVariable(c)}_{c\in \ClauseSet}$, i.e., every node in $G$ represents a variable and every hyperegde in $G$ represents the variable triplet of a clause. We consider simple \emph{Berge-cycles} (or, simply, cycles) of length $\ell \ge 2$. A cycle of length $\ell>2$ is defined as a set of $\ell$ distinct hyperedges $e_1, e_2, \dots, e_\ell \in E$ such that pairs of consecutive edges share exactly one vertex ($|e_i\cap e_{i+1}|=1$ and $|e_\ell\cap e_1|=1$) and all other pairs of hyperedges are disjoint. In a cycle of length $\ell=2$, the two hyperedges $e_1$ and  $e_2$ have at least two vertices in common. We similarly define a path of length $\ell\ge 2$, where pairs of consecutive edges share exactly one vertex ($|e_i\cap e_{i+1}|=1$ for $i\in [\ell-1]$) and all other pairs of hyperedges are disjoint. We observe that any unsatisfiable instance must contain a subgraph $H$ in $G$ such that every variable in $H$ appears in at least two hyperedges of $H$, which in turn implies that $G$ has a cycle. That translates into the following upper bound on the probability that instance $\ClauseSet\sim\Distribution$ is not satisfiable:
\begin{align}\nonumber
    \Prlong[\ClauseSet\sim\Distribution]{\neg SAT(\ClauseSet)}&\le
    \Prlong[\ClauseSet\sim\Distribution]{\exists \text{ cycle in } G(\ClauseSet)}\le
    \sum_{\ell\ge 2}\Exlong[\ClauseSet\sim\Distribution]{\cycle(G(\ClauseSet))}\\\nonumber
    &\le \Exlong[\ClauseSet\sim\Distribution]{\kpath(G(\ClauseSet))}
    +\sum_{\ell\ge 2}^{k-1}\Exlong[\ClauseSet\sim\Distribution]{\cycle(G(\ClauseSet))}\\\label{eq:prob_unsat_via_expect}
    &=
    \Exlong[\ClauseSet\sim\Distind]{\kpath(G(\ClauseSet))}+\sum_{\ell\ge 2}^{k-1}\Exlong[\ClauseSet\sim\Distind]{\cycle(G(\ClauseSet))}.
\end{align}
In the above derivation, $\cycle(G(\ClauseSet))$ counts the number of distinct cycles of length $\ell$ in $G$, and $\kpath(G(\ClauseSet))$ counts the number of distinct paths of lengths $k$ in $G$. For each ordered tuple of $k$  hyperedges in $G$ (there are $k!\cdot \binom{m}{k}$ such orderings), we can easily calculate the probability that they form a path. Thus,
\begin{equation}
    \label{eq:expect_path}
    \Exlong[\ClauseSet\sim\Distind]{\kpath(G(\ClauseSet))}=k!\cdot\binom{m}{k}\cdot\frac{3\cdot\binom{n-3}{2}}{\binom{n}{3}}\cdot\prod_{t=3}^{k}\frac{2\cdot\binom{n-2t+1}{2}}{\binom{n}{3}}\le
    m^k\cdot\frac{9}{n}\cdot\left(\frac{6}{n}\right)^{k-2}.
\end{equation}
For each ordering of $k$ edges, the first edge $e_1$ can be chosen arbitrarily; the second edge has exactly one vertex in common with $e_1$ and the remaining two vertices are chosen from $[n]\setminus e_1$ vertices; each of the remaining edges $e_t$ for $t\ge 3$ has exactly one of the two vertices of $e_{t-1}$ different from $e_{t-1}\cap e_{t-2}$ and has two other vertices chosen from $[n]\setminus(e_1\cup\ldots\cup e_{t-1})$. The upper bound follows after simplifying each of the terms. 
We similarly derive an upper bound on $\Ex{\cycle(G(\ClauseSet))}$ for $\ell\ge 3$ as follows:
\begin{equation*}
    \Exlong[\ClauseSet\sim\Distind]{\cycle(G(\ClauseSet))}\le\frac{1}{6}\cdot\ell!\cdot\binom{m}{\ell}\cdot\left[\frac{3\cdot\binom{n-3}{2}}{\binom{n}{3}}\cdot\prod_{t=3}^{\ell-1}\frac{2\cdot\binom{n-2t+1}{2}}{\binom{n}{3}}\right]\cdot\frac{4\cdot(n-2\ell+1)}{\binom{n}{3}},
\end{equation*}
where for each ordering of $\ell$ edges (now $2\ell\ge 6$ different orderings correspond to the same cycle), the probabilities to choose the first $\ell-1$ edges can be computed in the same way as for $\kpath$; for the last hyper-edge $e_{\ell}$, there are $4=2\cdot 2$ ways to select a vertex of $e_{\ell-1}$ together with a vertex of $e_1$, and $n-2\ell+1$ choices for the new vertex in $[n]\setminus(e_1\cup\ldots\cup e_{\ell-1})$. Hence,
\begin{equation}
\label{eq:expect_cycle}
    \Exlong[\ClauseSet\sim\Distind]{\cycle(G(\ClauseSet))}\le
    m^{\ell}\cdot\left[\frac{9}{n}\cdot\left(\frac{6}{n}\right)^{\ell-3}\right]\cdot\frac{4}{n^2}=\frac{1}{6}\left(\frac{6m}{n}\right)^{\ell}.
\end{equation}
For $\ell =2$ we have 
\begin{equation}
\label{eq:expect_two_cycle}
    \Exlong[\ClauseSet\sim\Distind]{\cycle[2](G(\ClauseSet))}=
    \binom{m}{2}\cdot\frac{3\cdot(n-3)+1}{\binom{n}{3}}=\frac{3 m\cdot (m-1)\cdot (3n-8)}{n\cdot (n-1)\cdot (n-2)}\le\left(\frac{3 m}{n}\right)^2.
\end{equation}
We conclude the proof by plugging estimates \eqref{eq:expect_path},\eqref{eq:expect_cycle}, and \eqref{eq:expect_two_cycle} into \eqref{eq:prob_unsat_via_expect}.
\begin{align*}
    \Prx{\neg SAT(\ClauseSet)} &\le\frac{2}{3}\frac{(6m)^k}{n^{k-1}}+
    \left(\frac{3 m}{n}\right)^2+\frac{1}{6}\sum_{\ell=3}^{k-1}\left(\frac{6 m}{n}\right)^{\ell}\\
    &\le\frac{2}{3\cdot 2^k}+
    \frac{1}{16}+\frac{1}{6}\cdot\left(\frac{1}{2^3}+\frac{1}{2^4}+\ldots+\frac{1}{2^{k-1}}\right)<\frac{1}{3}.  
\end{align*}
The second inequality follows since $m\le \frac{1}{12}\cdot n^{1-1/k}$.
\end{proof}

    \subsection{On the tightness of the $O(n^{1-1/k})$ bound for $\LST_k(n)$} 
Theorem~\ref{thm:lower-bound-k-wise-unsat} says that $k$-wise independence of $\Distribution\in\DistFamily_k(n,m)$ is enough to guarantee satisfiability of a random formula $\ClauseSet\sim\Distribution$ for the number of clauses $m$ of order $n^{1-1/k}$. On the other hand, for the mutually independent distribution $\Distind$ a random formula is unsatisfiable with high probability for $m=O(n)$. Furthermore, our analysis in Theorem~\ref{thm:lower-bound-k-wise-unsat} is essentially tight and it seems unlikely that there is a better bound than $m=O(n^{1-1/k})$. We discuss below why this is  the case, by outlining a plausible  way for constructing $k$-clause independent distribution $\Distribution\in\DistFamily_k(n,m)$ with $\Prx[\ClauseSet\sim\Distribution]{\neg SAT(C)}\ge \frac{2}{3}$ and $m=\Theta(n^{1-1/k})$. 

\paragraph{Informal outline of the construction.} In our construction of the $k$-clause independent distribution $\Distribution$ we are only concerned with the distribution of clauses over variables, as all literals will be assigned uniformly at random and independently. Then, a sufficient condition for a random formula $\ClauseSet\sim\Distribution$ to be unsatisfiable with large probability is that the $3$-uniform hypergraph $G(\ClauseSet)$ constructed in Theorem~\ref{thm:lower-bound-k-wise-unsat} has a dense subgraph, i.e., a subgraph on $|V(H)|$ vertices (corresponding to variables in $\ClauseSet$) with at least $|E(H)|\ge 100\cdot|V(H)|$ hyperedges.
Indeed, one can simply count the expected number of satisfying assignments of a random formula $\phi_{H}$ on $V(H)$ variables and $|E(H)|$ fixed clauses with randomly assigned literals: initially, all $2^{|V(H)|}$ assignment are satisfying, but then, as we add $|E(H)|$ clauses one-by-one, the expected number of satisfying assignments reduces each time exactly by a factor $7/8$ (each satisfying assignment disappears with probability $1/8$). At the end, we get that the expected number of satisfying assignments is $\left(\frac{7}{8}\right)^{100|V(H)|}\cdot 2^{|V(H)|}<0.01$, which means that $\phi$ is unsatisfiable most of the times. 

Now, we want to make sure that our hypergraph $G(\ClauseSet)$ often has such a ``dense'' subgraph $H$, to get an unsatisfiable random formula. Notice that $H$ needs only have a constant number of vertices. Also, 
note that according to the analysis in Theorem~\ref{thm:lower-bound-k-wise-unsat}, we need to avoid any cycles of length smaller than or equal to $k$, as the probability of having such a cycle is of order $o(1)$ in $G(\ClauseSet)$ for  $m=c\cdot n^{1-1/k}$ and any $k$-clause independent distribution $\ClauseSet\sim\DistFamily_k(n,m)$. Luckily, there are many construction of such $3$-uniform hypergraphs $H$ on $|V(H)|=O(1)$ vertices with large number of hyperedges $|E(H)|\ge 100|V(H)|$, and also of large girth $g(H)\ge k+1$ (e.g., see~\cite{EllisL14}). We shall ``plant'' $H$ (essentially insert $H$ as a connected component into $G$) with large probability in our construction. 
However, we need to make sure that by inserting $H$ into our graph $G$, we still have enough room to match $\Prlong{c_i=t_i, \forall i\in S}= \prod_{i\in S} \Prx{c_i=t_i}$ for all $S: |S|\le k$. Next, we give a high level idea how one could achieve this. 

As usual for the construction of distributions with identical marginals, we symmetrize $\Distribution$ over all possible permutations of variables in $\ClauseSet\sim\Distribution$. Then, our goal is to match the expected numbers for each   isomorphism class of configurations of $k$ hyperedges in $G(\ClauseSet)$ for $\ClauseSet\sim\Distribution$ with $\ClauseSet\sim\Distind$. It is 
useful to take note of the structure of typical hypergraph $G(\ClauseSet)$ for $\ClauseSet\sim\Distind$ when $m=c\cdot n^{1-1/k}$: it consists of connected components, each of which is a tree of size at most $k$; the number of connected components of size $k$ 
is a constant that grows slightly faster than linearly in $c$ and, more generally, the number of connected components of size $k-j$ is $\Theta(n^{j/k})$. There is also a $o(1)$ probability event of having a Berge-cycle or a connected component of size at least $k+1$ in $G(\ClauseSet)$ for $\ClauseSet\sim\Distind$. We can ignore in our 
construction of $\Distribution$ those events, by adding small probability mass to $\Distribution$ that consists of $\Distind$ conditional on any of these rare events (can be easily achieved via rejection sampling from $\Distind$). In this way, we 
only need to worry about matching expected numbers of forest configurations of size $k$ in $G(\ClauseSet)$ between $\ClauseSet\sim\Distribution$ and $\ClauseSet\sim\Distind$. We can pick the constant in $m=c\cdot n^{1-1/k}$ sufficiently large, so 
that the constant size subgraph $H$ contributes fewer trees of each type than their respective expected numbers in $G(\ClauseSet)$ for $\ClauseSet\sim\Distind$. Then, we can add a few more connected 
components that are trees of size $k$ to $\ClauseSet\sim\Distribution$, so that we match $k$-wise statistics on all tree configurations with $\ClauseSet\sim\Distind$. By having 
$H$ and a few trees of size $k$ in the graph $G(\ClauseSet)$ for $\ClauseSet\sim\Distribution$, we only utilize a constant number of variables.
Then, we would like to keep adding smaller connected components that consist of trees with strictly less than $k$ hyperedges and eventually match $k$-wise statistics on all forest-like configurations of size $k$.

It seems plausible that the approach outlined above should work and yield a $k$-clause independent distribution $\Distribution$. Apart from heavy notation that would be needed to formalize all steps in the above outline, the main technical hurdle is to ensure that statistics for all ``forest-like'' configurations of $k$ hyperedges with more than one connected component are perfectly matched with $G(\ClauseSet)$ for $\ClauseSet\sim\Distind$. 
Note, however, that even if our approach fails and there is a stronger version of Theorem~\ref{thm:lower-bound-k-wise-unsat} with $m=\omega(n^{1-1/k})$ and
$\Prx[\ClauseSet\sim\Distribution]{SAT(\ClauseSet)}\ge\frac{2}{3}$ for any $k$-clause 
independent distribution $\Distribution$, such a theorem would require a rather nontrivial argument that relies on subtle dependencies between multiple $k$-wise statistics.

\bibliographystyle{plain}
\bibliography{reference}

\begin{thebibliography}{10}

\bibitem{Achlioptas2000}
Dimitris Achlioptas.
\newblock Setting 2 variables at a time yields a new lower bound for random
  3-{SAT}.
\newblock In {\em Proceedings of the 32nd Annual {ACM} Symposium on Theory of
  Computing (STOC)}, pages 28--37, 2000.

\bibitem{A21}
Dimitris Achlioptas.
\newblock Random satisfiabiliy.
\newblock In {\em Handbook of Satisfiability}, volume 336 of {\em Frontiers in
  Artificial Intelligence and Applications}, pages 437--462. 2nd edition, 2021.

\bibitem{AP04}
Dimitris Achlioptas and Yuval Peres.
\newblock The threshold for random k-sat is $2klog2 - o(k)$.
\newblock {\em Journal of the American Mathematical Society}, 17:947--973,
  2004.

\bibitem{AAKMRX07}
Noga Alon, Alexandr Andoni, Tali Kaufman, Kevin Matulef, Ronitt Rubinfeld, and
  Ning Xie.
\newblock Testing \emph{k}-wise and almost \emph{k}-wise independence.
\newblock In {\em Proceedings of the 39th Annual {ACM} Symposium on Theory of
  Computing (STOC)}, pages 496--505, 2007.

\bibitem{AlonN08}
Noga Alon and Asaf Nussboim.
\newblock k-wise independent random graphs.
\newblock In {\em Proceedings of 49th Annual {IEEE} Symposium on Foundations of
  Computer Science (FOCS)}, pages 813--822, 2008.

\bibitem{AGL12}
Carlos Ans{\'{o}}tegui, Jes{\'{u}}s Gir{\'{a}}ldez{-}Cru, and Jordi Levy.
\newblock The community structure of {SAT} formulas.
\newblock In {\em Proceedings of 15th International Conference on Theory and
  Applications of Satisfiability Testing (SAT)}, pages 410--423, 2012.

\bibitem{BKTS02}
Paul Beame, Richard Karp, Toniann Pitassi, and Michael Saks.
\newblock The efficiency of resolution and davis--putnam procedures.
\newblock {\em SIAM J. Comput.}, 31(4):1048–1075, April 2002.

\bibitem{BenSassonW01}
Eli Ben-Sasson and Avi Wigderson.
\newblock Short proofs are narrow—resolution made simple.
\newblock {\em Journal of the ACM}, 48(2):149–169, 2001.

\bibitem{BenjaminiGP12}
Itai Benjamini, Ori Gurel-Gurevich, and Ron Peled.
\newblock On \emph{k}-wise independent distributions and boolean functions.
\newblock {\em arXiv:math}, abs/1201.3261, 2012.

\bibitem{BBCKW2001}
B{\'{e}}la Bollob{\'{a}}s, Christian Borgs, Jennifer~T. Chayes, Jeong~Han Kim,
  and David~Bruce Wilson.
\newblock The scaling window of the 2-{SAT} transition.
\newblock {\em Random Structures and Algorithms}, 18(3):201--256, 2001.

\bibitem{Bringmann14}
Karl Bringmann.
\newblock Why walking the dog takes time: Frechet distance has no strongly
  subquadratic algorithms unless {SETH} fails.
\newblock In {\em Proceedings of 55th {IEEE} Annual Symposium on Foundations of
  Computer Science, (FOCS)}, pages 661--670, 2014.

\bibitem{BFU1993}
Andrei~Z. Broder, Alan~M. Frieze, and Eli Upfal.
\newblock On the satisfiability and maximum satisfiability of random 3-{CNF}
  formulas.
\newblock In {\em Proceedings of the 4th Annual {ACM-SIAM} Symposium on
  Discrete Algorithms (SODA)}, pages 322--330, 1993.

\bibitem{CGLW21}
Ioannis Caragiannis, Nick Gravin, Pinyan Lu, and Zihe Wang.
\newblock Relaxing the independence assumption in sequential posted pricing,
  prophet inequality, and random bipartite matching.
\newblock In {\em Proceedings of 17th International Conference on Web and
  Internet Economics (WINE)}, pages 131--148, 2021.

\bibitem{CKT1991}
Peter~C. Cheeseman, Bob Kanefsky, and William~M. Taylor.
\newblock Where the really hard problems are.
\newblock In {\em Proceedings of the 12th International Joint Conference on
  Artificial Intelligence (IJCAI)}, pages 331--340, 1991.

\bibitem{CR1992}
Vasek Chv{\'{a}}tal and Bruce~A. Reed.
\newblock Mick gets some (the odds are on his side).
\newblock In {\em Proceedings of 33rd Annual Symposium on Foundations of
  Computer Science (FOCS)}, pages 620--627, 1992.

\bibitem{ChvatalS88}
Va\v{s}ek Chv\'{a}tal and Endre Szemer\'{e}di.
\newblock Many hard examples for resolution.
\newblock {\em Journal of the ACM}, 35(4):759–768, 1988.

\bibitem{Coja14}
Amin Coja-Oghlan.
\newblock The asymptotic \emph{k}-sat threshold.
\newblock In {\em Proceedings of the 46th Annual ACM Symposium on Theory of
  Computing (STOC)}, pages 804--813, 2014.

\bibitem{CNPPRW22}
Marek Cygan, Jesper Nederlof, Marcin Pilipczuk, Michal Pilipczuk, Johan M.~M.
  van Rooij, and Jakub~Onufry Wojtaszczyk.
\newblock Solving connectivity problems parameterized by treewidth in single
  exponential time.
\newblock {\em ACM Transactions on Algorithms}, 18(2):17:1--17:31, 2022.

\bibitem{DKMP09}
Josep D{\'{\i}}az, Lefteris~M. Kirousis, Dieter Mitsche, and Xavier
  P{\'{e}}rez{-}Gim{\'{e}}nez.
\newblock On the satisfiability threshold of formulas with three literals per
  clause.
\newblock {\em Theoretical Computer Science}, 410(30-32):2920--2934, 2009.

\bibitem{DSS2022}
Jian Ding, Allan Sly, and Nike Sun.
\newblock {Proof of the satisfiability conjecture for large $k$}.
\newblock {\em Annals of Mathematics}, 196(1):1--388, 2022.

\bibitem{DBM2000}
Olivier Dubois, Yacine Boufkhad, and Jacques Mandler.
\newblock Typical random 3-{SAT} formulae and the satisfiability threshold.
\newblock In {\em Proceedings of the 11th Annual {ACM-SIAM} Symposium on
  Discrete Algorithms (SODA)}, pages 126--127, 2000.

\bibitem{DKP24}
Shaddin Dughmi, Yusuf~Hakan Kalayci, and Neel Patel.
\newblock Limitations of stochastic selection problems with pairwise
  independent priors.
\newblock In {\em Proceedings of the 56th Annual {ACM} Symposium on Theory of
  Computing (STOC)}, pages 479--490, 2024.

\bibitem{EllisL14}
David Ellis and Nathan Linial.
\newblock On regular hypergraphs of high girth.
\newblock {\em The Electronic Journal of Combinatorics}, 21(1):1, 2014.

\bibitem{Feige02}
Uriel Feige.
\newblock Relations between average case complexity and approximation
  complexity.
\newblock In {\em Proceedings of the Thiry-Fourth Annual ACM Symposium on
  Theory of Computing}, STOC '02, page 534–543, New York, NY, USA, 2002.
  Association for Computing Machinery.

\bibitem{FeigeO07}
Uriel Feige and Eran Ofek.
\newblock Easily refutable subformulas of large random 3cnf formulas.
\newblock {\em Theory of Computing}, 3(1):25--43, 2007.

\bibitem{Flaxman16}
Abraham Flaxman.
\newblock Random planted $3$-sat.
\newblock In Ming-Yang Kao, editor, {\em Encyclopedia of Algorithms}, pages
  1728--1732. Springer New York, New York, NY, 2016.

\bibitem{FP1983}
John Franco and Marvin~C. Paull.
\newblock Probabilistic analysis of the davis putnam procedure for solving the
  satisfiability problem.
\newblock {\em Discrete Applied Mathematics}, 5(1):77--87, 1983.

\bibitem{FriedmanGK05}
Joel Friedman, Andreas Goerdt, and Michael Krivelevich.
\newblock Recognizing more unsatisfiable random \emph{k}-{SAT} instances
  efficiently.
\newblock {\em SIAM Journal on Computing}, 35(2):408–430, 2005.

\bibitem{FKRRS2017}
Tobias Friedrich, Anton Krohmer, Ralf Rothenberger, Thomas Sauerwald, and
  Andrew~M. Sutton.
\newblock Bounds on the satisfiability threshold for power law distributed
  random {SAT}.
\newblock In {\em Proceedings of 25th Annual European Symposium on Algorithms
  (ESA)}, pages 37:1--37:15, 2017.

\bibitem{FR2018}
Tobias Friedrich and Ralf Rothenberger.
\newblock Sharpness of the satisfiability threshold for non-uniform random
  \emph{k}-{SAT}.
\newblock In {\em Proceedings of 21st International Conference on Theory and
  Applications of Satisfiability Testing (SAT)}, pages 273--291, 2018.

\bibitem{FR2019}
Tobias Friedrich and Ralf Rothenberger.
\newblock The satisfiability threshold for non-uniform random 2-{SAT}.
\newblock In {\em Proceedings of 46th International Colloquium on Automata,
  Languages, and Programming, (ICALP)}, pages 61:1--61:14, 2019.

\bibitem{FS1996}
Alan~M. Frieze and Stephen Suen.
\newblock Analysis of two simple heuristics on a random instance of
  \emph{k}-{SAT}.
\newblock {\em Journal of Algorithms}, 20(2):312--355, 1996.

\bibitem{GW24}
Nick Gravin and Zhiqi Wang.
\newblock On robustness to \emph{k}-wise independence of optimal bayesian
  mechanisms.
\newblock {\em CoRR}, abs/2409.08547, 2024.

\bibitem{GHKL24}
Anupam Gupta, Jinqiao Hu, Gregory Kehne, and Roie Levin.
\newblock Pairwise-independent contention resolution.
\newblock In {\em Proceddings of 25th International Conference on Integer
  Programming and Combinatorial Optimization (IPCO)}, pages 196--209, 2024.

\bibitem{HS03}
Mohammad~Taghi Hajiaghayi and Gregory~B Sorkin.
\newblock The satisfiability threshold of random 3-{SAT} is at least 3.52.
\newblock {\em arXiv:math}, abs/0310193, 2003.

\bibitem{IP01}
Russell Impagliazzo and Ramamohan Paturi.
\newblock On the complexity of k-sat.
\newblock {\em Journal of Computer and System Sciences}, 62(2):367--375, 2001.

\bibitem{IPZ01}
Russell Impagliazzo, Ramamohan Paturi, and Francis Zane.
\newblock Which problems have strongly exponential complexity?
\newblock {\em Journal of Computer and System Sciences}, 63(4):512--530, 2001.

\bibitem{KMPS95}
Anil Kamath, Rajeev Motwani, Krishna~V. Palem, and Paul~G. Spirakis.
\newblock Tail bounds for occupancy and the satisfiability threshold
  conjecture.
\newblock {\em Random Structures and Algorithms}, 7(1):59--80, 1995.

\bibitem{KKL06}
Alexis~C. Kaporis, Lefteris~M. Kirousis, and Efthimios~G. Lalas.
\newblock The probabilistic analysis of a greedy satisfiability algorithm.
\newblock {\em Random Structures and Algorithms}, 28(4):444–480, 2006.

\bibitem{MPZ02}
M.~Mezard, G.~Parisi, and R.~Zecchina.
\newblock Analytic and algorithmic solution of random satisfiability problems.
\newblock {\em Science}, 297(5582):812--815, 2002.

\bibitem{MSL1992}
David~G. Mitchell, Bart Selman, and Hector~J. Levesque.
\newblock Hard and easy distributions of {SAT} problems.
\newblock In {\em Proceedings of the 10th National Conference on Artificial
  Intelligence (AAAI)}, pages 459--465, 1992.

\bibitem{OB2019}
Oleksii Omelchenko and Andrei~A. Bulatov.
\newblock Satisfiability threshold for power law random 2-sat in configuration
  model.
\newblock In {\em Proceedings of 22nd International Conference on Theory and
  Applications of Satisfiability Testing (SAT)}, pages 53--70, 2019.

\bibitem{RX10}
Ronitt Rubinfeld and Ning Xie.
\newblock Testing non-uniform \emph{k}-wise independent distributions over
  product spaces.
\newblock In {\em Proceedings of 37th International Colloquium on Automata,
  Languages, and Programming (ICALP)}, pages 565--581, 2010.

\end{thebibliography}


\end{document}